\documentclass[12pt]{amsart}
\usepackage{amsmath}
\usepackage{amsthm}
\usepackage{amsfonts}
\usepackage{amssymb}
\usepackage{pgf,tikz}
\usetikzlibrary{decorations.pathreplacing,angles,quotes}
\usepackage{graphicx}
\usepackage{hyperref}

\newcommand*{\darr}{ {\downarrow} }
\newcommand*{\uarr}{ {\uparrow} }
\newcommand{\meet}{\land}
\newcommand{\join}{\lor}
\newcommand{\impla}{\rightarrow}
\newcommand{\nega}{\neg}
\newcommand{\xn}{\{ x_n \}}
\newcommand{\zn}{\{ z_n \}}
\newcommand{\lifm}{L\times_I^F M}
\newcommand{\xduy}{X \oplus_D^U Y}
\newcommand{\xdy}{X \oplus_D Y}

\newtheorem{theorem}{Theorem}[section]
\newtheorem*{theorem*}{Theorem}
\newtheorem{lemma}[theorem]{Lemma}

\newtheorem*{proposition*}{Proposition}
\newtheorem{corollary}[theorem]{Corollary}
\newtheorem*{corollary*}{Corollary}

\theoremstyle{definition}
\newtheorem{definition}[theorem]{Definition}
\newtheorem*{definition*}{Definition}
\newtheorem{example}[theorem]{Example}
\newtheorem*{example*}{Example}

\newtheorem{remark}[theorem]{Remark}
\newtheorem*{remark*}{Remark}
\newtheorem{notation}[theorem]{Notation}

\begin{document}

\title[Characterization of metrizable Esakia spaces]{Characterization of metrizable Esakia spaces via some forbidden configurations}

\author[G. Bezhanishvili]{Guram Bezhanishvili}
\address{Department of Mathematical Sciences\\
New Mexico State University\\
Las Cruces NM 88003\\
USA}
\email{guram@nmsu.edu}

\author[L. Carai]{Luca Carai}
\address{Department of Mathematical Sciences\\
New Mexico State University\\
Las Cruces NM 88003\\
USA}
\email{lcarai@nmsu.edu}



\begin{abstract}
By Priestley duality, each bounded distributive lattice is represented as the lattice of clopen upsets of a Priestley space,
and by Esakia duality, each Heyting algebra is represented as the lattice of clopen upsets of an Esakia space.
Esakia spaces are those Priestley spaces that satisfy the additional condition that the downset of each clopen is clopen.
We show that in the metrizable case Esakia spaces can be singled out by forbidding three simple configurations.
Since metrizability yields that the corresponding lattice of clopen upsets is countable, this provides a characterization
of countable Heyting algebras. We show that this characterization no longer holds in the uncountable case.
Our results have analogues for co-Heyting algebras and bi-Heyting algebras, and they easily generalize to the setting of p-algebras.
\end{abstract}

\maketitle

\section{Introduction}

Priestley duality \cite{pri70,pri72} provides a dual equivalence between the category $\sf Dist$ of bounded distributive lattices and the category
$\sf Pries$ of Priestley spaces; and Esakia duality \cite{esa74} provides a dual equivalence between the category $\sf Heyt$ of Heyting algebras
and the category $\sf Esa$ of Esakia spaces.
To make the paper self-contained, we recall main definitions.

An \textit{ordered topological space} is a triple $(X, \mathcal{T}, \leq)$
such that $(X, \mathcal{T})$ is a topological space and $\leq$ is a partial order on $X$. When we say that an ordered topological space is
compact, metrizable, etc.~we mean that the underlying topological space is compact, metrizable, etc.
As usual, for $A \subseteq X$ we let
\[
\uarr A =\{ x \in X \mid a \leq x \text{ for some } a \in A \}
\]
and
\[
\darr A = \{ x \in X \mid x \leq a \text{ for some } a \in A \}.
\]
If $A= \{x\}$, then we write $\uarr x$ and $\darr x$, respectively.
We call $A$ an \emph{upset} if $\uarr A=A$ and a \emph{downset} if $\darr A=A$.

\begin{definition}
\begin{enumerate}
\item[]
\item An ordered topological space $(X, \mathcal{T}, \leq)$ satisfies the \textit{Priestley separation axiom} if
$x \nleq y$ implies that there is a clopen upset $U$ such that $x \in U$ and $y \notin U$.
\item A \textit{Priestley space} is an ordered topological space that is compact and satisfies the Priestley separation axiom.
\end{enumerate}
\end{definition}

\begin{notation}
To simplify notation, we will suppress $\mathcal T$ and $\le$ and denote a Priestley space simply by $X$.
\end{notation}

\begin{remark}
The following facts about Priestley spaces are well known:
\begin{enumerate}
\item Each Priestley space is a Stone space (compact, Hausdorff, zero-dimen\-sional space).
\item If $F$ is closed, then so are $\uarr F$ and $\darr F$.
\item There exist Priestley spaces such that the downset or upset of an open set may not be open.
\end{enumerate}
\end{remark}

The Priestley space of a bounded distributive lattice $L$ is constructed by taking the set $X$ of prime filters of $L$, the order on $X$
is the inclusion order, and the topology on $X$ is given by the basis
\[
\{\alpha(a)\setminus\alpha(b)\mid a,b\in L\}
\]
where
\[
\alpha(a)=\{x\in X\mid a\in x\}.
\]
Then $\alpha$ is an isomorphism of $L$ onto the lattice of clopen upsets of $X$.

\begin{definition}
A Priestley space is an \textit{Esakia space} if the downset of each open set is open (equivalently, the downset of each clopen set is clopen).
\end{definition}

\begin{remark}
In an Esakia space, the upset of an open set may not be open.
\end{remark}

Heyting algebras are the bounded distributive lattices $L$ with an additional binary operation $\to$ of relative pseudo-complement
which satisfies, for all $a,b,x\in L$:
\[
a\wedge x\le b \mbox{ iff } x\le a\to b.
\]
It turns out that the lattice of clopen upsets of a Priestley space $X$ is a Heyting algebra iff it is an Esakia space,
where the relative pseudo-complement of two clopen upsets $U,V$ is given by $X \setminus \darr (U \setminus V)$.

The three spaces $Z_1$, $Z_2$, and $Z_3$ depicted in Figure~\ref{fig:3forbidden} are probably the simplest examples of Priestley spaces
that are not Esakia spaces. Topologically each of the three spaces is homeomorphic to the one-point compactification of the countable
discrete space $\{y\}\cup\{z_n\mid n\in\omega\}$, with $x$ being the limit point of $\{z_n\mid n\in\omega\}$. For each of the three
spaces, it is straightforward to check that with the partial order whose Hasse diagram is depicted in Figure~\ref{fig:3forbidden},
the space is a Priestley space. On the other hand, neither of the three spaces is an Esakia space because $\{y\}$ is open, but
${\downarrow}y=\{x,y\}$ is no longer open.

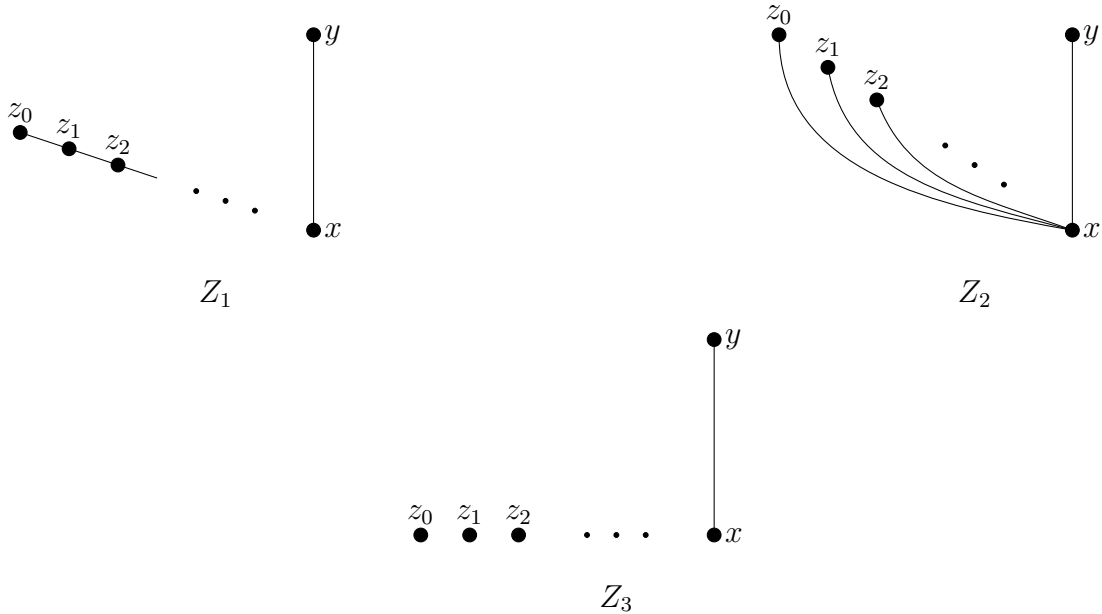
\begin{figure}[!h]
\begin{center}
\begin{tikzpicture}[scale=1.3]
\draw [fill=black] (0,0) circle (2.0pt);
\draw [fill=black] (0,2) circle (2.0pt);
\draw [fill=black] (-2,2/3) circle (2.0pt);
\draw [fill=black] (-2.5,2.5/3) circle (2.0pt);
\draw [fill=black] (-3,1) circle (2.0pt);
\node [right] at (0,0) {$x$};
\node [right] at (0,2) {$y$};
\node [above] at (-2,2/3) {$z_2$};
\node [above] at (-2.5,2.5/3) {$z_1$};
\node [above] at (-3,1) {$z_0$};
\draw (0,0) -- (0,2);
\draw (-1.6,1.6/3) -- (-2,2/3) -- (-2.5,2.5/3) -- (-3,1);
\draw [fill=black] (-0.6,0.6/3) circle (0.7pt);
\draw [fill=black] (-0.9,0.9/3) circle (0.7pt);
\draw [fill=black] (-1.2,1.2/3) circle (0.7pt);
\node [below] at (-1,-0.4) {$Z_1$};
\end{tikzpicture}
\hspace{2in}
\begin{tikzpicture}[scale=1.3]
\draw [fill=black] (0,0) circle (2.0pt);
\draw [fill=black] (0,2) circle (2.0pt);
\draw [fill=black] (-2,2*2/3) circle (2.0pt);
\draw [fill=black] (-2.5,2.5*2/3) circle (2.0pt);
\draw [fill=black] (-3,3*2/3) circle (2.0pt);
\node [right] at (0,0) {$x$};
\node [right] at (0,2) {$y$};
\node [above] at (-2,2*2/3) {$z_2$};
\node [above] at (-2.5,2.5*2/3) {$z_1$};
\node [above] at (-3,3*2/3) {$z_0$};
\draw (0,0) -- (0,2);
\draw (0,0) to [out=160,in=290] (-2,2*2/3);
\draw (0,0) to [out=165,in=280] (-2.5,2.5*2/3);
\draw (0,0) to [out=170,in=270] (-3,3*2/3);
\draw [fill=black] (-0.7,0.7*2/3) circle (0.7pt);
\draw [fill=black] (-1,1*2/3) circle (0.7pt);
\draw [fill=black] (-1.3,1.3*2/3) circle (0.7pt);
\node [below] at (-1,-0.4) {$Z_2$};
\end{tikzpicture}
\newline
\begin{tikzpicture}[scale=1.3]
\draw [fill=black] (0,0) circle (2.0pt);
\draw [fill=black] (0,2) circle (2.0pt);
\draw [fill=black] (-2,0) circle (2.0pt);
\draw [fill=black] (-2.5,0) circle (2.0pt);
\draw [fill=black] (-3,0) circle (2.0pt);
\node [right] at (0,0) {$x$};
\node [right] at (0,2) {$y$};
\node [above] at (-2,0) {$z_2$};
\node [above] at (-2.5,0) {$z_1$};
\node [above] at (-3,0) {$z_0$};
\draw (0,0) -- (0,2);
\draw [fill=black] (-0.7,0) circle (0.7pt);
\draw [fill=black] (-1,0) circle (0.7pt);
\draw [fill=black] (-1.3,0) circle (0.7pt);
\node [below] at (-1,-0.4) {$Z_3$};
\end{tikzpicture}
\end{center}
\caption{The three Priestley spaces $Z_1$, $Z_2$, and $Z_3$.}
\label{fig:3forbidden}
\end{figure}

In this paper we show that a metrizable Priestley space is not an Esakia space exactly when one of these three spaces can be embedded in it.
The embeddings we consider are special in that the point $y$ plays a special role. We show that this condition on the embeddings, as well as
the metrizability condition, cannot be dropped by presenting some counterexamples. In doing so, we develop a way to combine two Priestley spaces
which has proved to be useful in building Priestley spaces that are not Esakia spaces. An advantage of our characterization lies in the fact
that when a metrizable Priestley space $X$ is presented by a Hasse diagram, it is easy to verify whether or not $X$ contains one of the three
``forbidden configurations".

The paper is organized as follows. In Section~2 we present the main result by showing that a metrizable Priestley space is not an Esakia space
iff a copy of one of the three forbidden configurations sits inside $X$ in a special way. In Section~3 we translate our main result
into the dual lattice-theoretic statement, yielding a characterization of countable Heyting algebras. This characterization easily generalizes
to the setting of p-algebras, and also has analogues for co-Heyting and bi-Heyting algebras. In Section~4 we present the ``down-up sum" of
Priestley spaces, and its dual ``ideal-filter product" of lattices. Finally, Section~5 is devoted to counterexamples. We use the down-up sum
to build non-metrizable Priestley spaces that are not Esakia spaces
and yet do not
contain a copy of any of the three forbidden configurations. This shows that there is no obvious generalization of our results to the
non-metrizable setting. We finish by showing that the additional condition on the embeddings cannot be dropped either.

\section{The main theorem}

\begin{definition} \label{def:forbidconf}
Let $X$ be a Priestley space. We say that $Z_i$ ($i=1,2,3$) is a \emph{forbidden configuration} for $X$ if there are a topological and order
embedding $e \colon Z_i\to X$ and an open neighborhood $U$ of $e(y)$ such that $e^{-1}({\downarrow}U)=\{x,y\}$.
\end{definition}

The next result shows that whether a metrizable Priestley space is an Esakia space is determined by these three forbidden configurations.
The key assumption of metrizability is used to show that if $x$ is a limit point of a set, then there is a sequence in the set converging to $x$.
This can be done already for the Priestley spaces that are sequential spaces (see Remark~\ref{rem:sequential}). The necessity of the sequentiality
assumption will be discussed in more detail in Section~\ref{section:counterexamples}.

\begin{theorem} \label{thm:maintheorem}
A metrizable Priestley space $X$ is not an Esakia space iff one of $Z_1, Z_2, Z_3$ is a forbidden configuration for $X$.
\end{theorem}

\begin{proof}
First suppose that one of the $Z_i$ is a forbidden configuration for $X$. Since $e \colon Z_i\to X$ is continuous and $e^{-1}(\darr U)= \{ x,y \}$
is not open in $Z_i$, we conclude that $\darr U$ is not open in $X$. Thus, $X$ is not an Esakia space.

Conversely, suppose that $X$ is not an Esakia space. Then there is an
open
subset $U$ of $X$ such that $\darr U$ is not open.
Therefore, $(\darr U)^c$ is not closed. Since $X$ is metrizable, there is a sequence $\xn \subseteq (\darr U)^c$ such that
$\lim x_n =x \in \darr U$. As $X$ is Hausdorff, $\xn$ has to be infinite, hence we may assume that $x_n \ne x_m$ for $n \ne m$.
Because $U$ is open, we have $x \in \darr U \setminus U$. Let $y \in U$ be such that $x \leq y$. Then $x < y$.

Observe that $x_n \nleq x$ and $x_n \nleq y$ for any $n$ because otherwise $x_n \in \darr U$.
In addition, if there is $M$ such that $y \leq x_n$ for all $n \geq M$, then $x_n \in \uarr y$ for all $n \geq M$. Since $\uarr y$ is
closed and $x = \lim x_n$, this would yield $x \in \uarr y$, a contradiction. Therefore, $y \nleq x_n$ for some $n \geq M$. Thus, we
can select a subsequence of $\{ x_n \}$ each member of which is not above $y$. Hence, we may assume without loss of generality that
$x_n$ and $y$ are incomparable for all $n$. We now have two cases to consider.

{\bf Case 1:} There is an infinite subsequence $\{y_n\}$ of $\xn$ that is totally ordered by $\le$.
Since $\{y_n\}$ is an infinite subsequence of $\xn$, we have $\lim y_n = x$.
Consider the closure $\overline{\{y_n\}}$. As $\{y_n\}$ is totally ordered, by \cite[Lem.~3.1]{bm11}, $\overline{\{y_n\}}$ is also totally
ordered and has max and min. Since $x \in \overline{\{y_n\}}$ which is totally ordered, for each $n$ we have $x \leq y_n$ or $y_n \leq x$.
But, as we already observed, $y_n \nleq x$. Thus, $x \leq y_n$ for each $n$. Since $x \notin (\darr U)^c$, we have $x < y_n$. We now define
recursively a subsequence $\zn$ of $\{y_n\}$ such that $z_0> z_1 > z_2 > \cdots$.

Set $z_0=y_0$. If $z_k=y_{n_k}$ is already defined, then since $\lim y_n=x$ and $x < y_{n_k}$, there is a clopen downset $V$ of $X$
such that $x \in V$, $y_{n_k}=z_k \notin V$, and $V$ contains an infinite subset of $\{y_n\}$. So there is $y_{n_{k+1}} \in V$ such that
$n_{k+1}> n_k$. Therefore, $y_{n_{k+1}} < y_{n_k}$. Set $z_{k+1}=y_{n_{k+1}}$. We thus obtain a sequence $z_0> z_1 > z_2 > \cdots$ such that
$\lim z_n =x$ and each $z_n$ is incomparable with $y$.

Let $Z=\{ y,x \} \cup \zn \subseteq X$, and view $Z$ as an ordered topological space with the order and topology inherited from $X$.
Since $Z \cap U= \{ y \}$ and $\{y \}$ is closed in $X$, we have that $\{ y \}$ is clopen in $Z$. For each $m$, we show that
the singleton $\{ z_m \}$ is clopen in $Z$. As $x < z_m$, there is a clopen downset $V$ of $X$ such that $x \in V$ and
$z_m \notin V$, so $V^c \cap Z$ is finite and contains $z_m$. Since $X$ is Hausdorff, so is $V^c \cap Z$.
Because every finite Hausdorff space is discrete, $\{ z_m \}$ is clopen in $V^c \cap Z$, which is clopen in $Z$. Thus, the
singleton $\{ z_m \}$ is clopen in $Z$.

Opens in $Z$ containing $x$ are exactly the cofinite subsets of $Z$ because $\lim z_n=x$ and all the singletons except $\{ x \}$ are clopen.
Therefore, $Z$ is order-isomorphic and homeomorphic to the Priestley space $Z_1$.

{\bf Case 2:} There is no infinite totally ordered subsequence of $\xn$. Since every infinite poset contains either an infinite chain
or an infinite antichain (see, e.g., \cite[Thm.~1.14]{rom08}),
there is an infinite subsequence $\{y_n\}$ of $\xn$ that is an antichain. As $\{y_n\}$ is an infinite subsequence of $\xn$,
we have that $\lim y_n=x$.
Our goal is to select a subsequence $\{z_n \}$ of $\{ y_n \}$ so that either $Z_2$ or $Z_3$ becomes a forbidden configuration.
Which of the two becomes a forbidden configuration depends on whether or not ${\uparrow} x \cap \{ y_n \}$ is infinite.

{\bf Case 2a:} ${\uparrow} x \cap \{y_n\}$ is infinite. Then $\zn := {\uparrow} x \cap \{y_n\}$ is an infinite subsequence of $\{y_n\}$ such
that $\lim z_n = x$ and each $z_n$ is incomparable with $y$. Let $Z=\{ y,x \} \cup \zn \subseteq X$, and view $Z$ as an ordered topological
space with the order and topology inherited from $X$. Since $x < z_m$ for each $m$, by arguing as in Case 1 we obtain that $Z$ is order-isomorphic
and homeomorphic to the Priestley space $Z_2$.

{\bf Case 2b:} ${\uparrow} x \cap \{ y_n \}$ is finite. Then $\zn := ({\uparrow} x)^c \cap \{y_n\}$ is an infinite subsequence of $\{y_n\}$
such that $\lim z_n =x$ and each $z_n$ is incomparable with $y$. Let $Z=\{ y,x \} \cup \zn \subseteq X$, and view $Z$ as an ordered topological
space with the order and topology inherited from $X$. Since $x$ and $z_m$  are incomparable for each $m$, by arguing as in Case 1 we obtain
that $Z$ is order-isomorphic and homeomorphic to the Priestley space $Z_3$.
\end{proof}

\begin{remark}\label{rem:sequential}
In the proof of Theorem~\ref{thm:maintheorem} metrizability was used to find in a set that is not closed a sequence
converging outside of it. We recall (see, e.g., \cite[p.~53]{eng89}) that a topological space $X$ is a \textit{sequential space}
provided a set $A$ is closed in $X$ iff together with each sequence $A$ contains all its limits.
Thus, Theorem~\ref{thm:maintheorem} holds not only for metrizable Priestley spaces, but more generally, for sequential Priestley spaces.
\end{remark}

\section{Algebraic meaning of the result}

Let $L_1$, $L_2$, and $L_3$ be the dual lattices of $Z_1$, $Z_2$, and $Z_3$, respectively.
Clopen upsets of $Z_1$ are the whole space, the empty set, ${\uparrow} z_n$, and ${\uparrow} z_n \cup \{y\}$ for $n\in\omega$.
Thus, $L_1$ can be depicted as in Figure~\ref{fig:L1}. Note that $L_1$ is not a Heyting algebra since $\nega c$ does not exist.

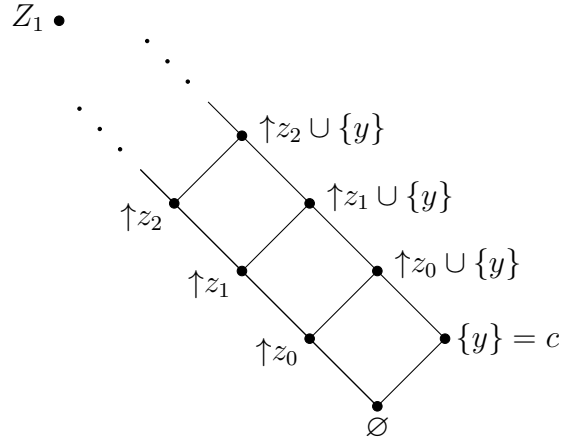
\begin{figure}[!h]
\begin{center}
\begin{tikzpicture}[scale=0.9]
\draw [fill=black] (0,0) circle (2.0pt);
\draw [fill=black] (-1,1) circle (2.0pt);
\draw [fill=black] (0,2) circle (2.0pt);
\draw [fill=black] (1,1) circle (2.0pt);
\draw [fill=black] (-1,3) circle (2.0pt);
\draw [fill=black] (-2,2) circle (2.0pt);
\draw [fill=black] (-2,4) circle (2.0pt);
\draw [fill=black] (-3,3) circle (2.0pt);
\draw [fill=black] (-4.7,5.7) circle (2.0pt);
\draw (0,0) -- (-1,1) -- (-2,2) -- (-3,3) -- (-3.5,3.5) ;
\draw (1,1) -- (0,2) -- (-1,3) -- (-2,4) -- (-2.5,4.5) ;
\draw (0,0) -- (-1,1) -- (-2,2) -- (-3,3) -- (-3.5,3.5) ;
\draw (0,0) -- (1,1);
\draw (-1,1) -- (0,2);
\draw (-2,2) -- (-1,3);
\draw (-3,3) -- (-2,4);
\draw [fill=black] (-3.8,3.8) circle (0.7pt);
\draw [fill=black] (-4.1,4.1) circle (0.7pt);
\draw [fill=black] (-4.4,4.4) circle (0.7pt);
\draw [fill=black] (-2.8,4.8) circle (0.7pt);
\draw [fill=black] (-3.1,5.1) circle (0.7pt);
\draw [fill=black] (-3.4,5.4) circle (0.7pt);
\node [right] at (1,1) {$\{y \}=c$};
\node [below] at (0,0) {$\varnothing$};
\node [right] at (0.1,2.1) {$\uarr z_0 \cup \{ y \}$};
\node [right] at (-0.9,3.1) {$\uarr z_1 \cup \{ y \}$};
\node [right] at (-1.9,4.1) {$\uarr z_2 \cup \{ y \}$};
\node [left] at (-1,0.8) {$\uarr z_0$};
\node [left] at (-2,1.8) {$\uarr z_1$};
\node [left] at (-3,2.8) {$\uarr z_2$};
\node [left] at (-4.75,5.75) {$Z_1$};
\end{tikzpicture}
\end{center}
\caption{The lattice $L_1$.}
\label{fig:L1}
\end{figure}

Clopen upsets of $Z_2$ are the whole space, the empty set, and the finite subsets of $\{ y\}\cup\{z_n\mid n\in\omega \}$.
Therefore, $L_2$ is isomorphic to the lattice of finite subsets of $\omega$ together with a top element; see Figure~\ref{fig:L2}.
Thus, $L_2$ is not a Heyting algebra because $\nega F$ does not exist for any finite subset $F$ of $\omega$.

\begin{figure}[!h]
\begin{center}
\begin{tikzpicture}[scale=0.9]
\draw (-2,0) arc [radius=2, start angle=180, end angle=360];
\begin{scope}[yscale=0.7]
        \draw (-2,0) to [out=45, in=180] (-1.6,0.25);
        \draw (-1.6,0.25) to [out=0, in=135] (-1.2,0);
        \draw (-1.2,0) to [out=315, in=180] (-0.8,-0.25);
        \draw (-0.8,-0.25) to [out=0, in=45] (-0.4,0);
        \draw (-0.4,0) to [out=45, in=180] (0,0.25);
        \draw (0,0.25) to [out=0, in=135] (0.4,0);
        \draw (0.4,0) to [out=315, in=180] (0.8,-0.25);
        \draw (0.8,-0.25) to [out=0, in=45] (1.2,0);
        \draw (1.2,0) to [out=45, in=180] (1.6,0.25);
        \draw (1.6,0.25) to [out=0, in=135] (2,0);
    \end{scope}
\draw [fill=black] (0,2) circle (2.0pt);
\draw [fill=black] (0,-2) circle (2.0pt);
\draw [fill=black] (-1.5,-0.7) circle (1.3pt);
\draw (-2,0) -- (0,2);
\draw (2,0) -- (0,2);
\node [below] at (0.18,-0.5) {$\mathcal{P}_{\rm fin}(\omega)$};
\node [below] at (0,-2) {$\varnothing$};
\node [right] at (-1.5,-0.6) {$F$};
\node [right] at (0,2.2) {$Z_2$};
\end{tikzpicture}
\end{center}
\caption{The lattice $L_2$.}
\label{fig:L2}
\end{figure}
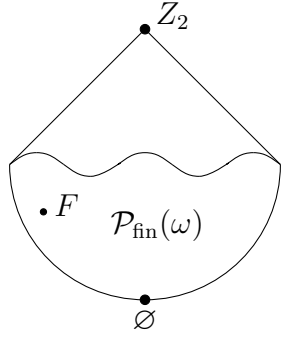

Clopen upsets of $Z_3$ are the whole space, the empty set, finite subsets of $\{y\}\cup\{z_n\mid n\in\omega \}$, and
$\{x,y \} \cup C$ where $C$ is a cofinite subset of $\{z_n\mid n\in\omega \}$.
Therefore, if we denote by ${\bf CF}(\omega)$ the Boolean algebra of finite and cofinite subsets of $\omega$ and by $\bf 2$
the two-element Boolean algebra, then $L_3$ is isomorphic to the sublattice of ${\bf CF}(\omega) \times {\bf 2}$ given by
the elements of the form $(A,n)$ where $A$ is finite or $n=1$; see Figure~\ref{fig:L3}. Thus, $L_3$ is not a Heyting algebra
because $\nega (F,1)$ does not exist for any finite $F$.

\begin{figure}[!h]
\begin{center}
\hspace*{-0.7cm}
\begin{tikzpicture}[scale=0.9]
\begin{scope}[shift={(-3,-2.5)}]
    \draw (-2,0) arc [radius=2, start angle=180, end angle=360];
    \begin{scope}[yscale=0.7]
           \draw (-2,0) to [out=45, in=180] (-1.6,0.25);
            \draw (-1.6,0.25) to [out=0, in=135] (-1.2,0);
            \draw (-1.2,0) to [out=315, in=180] (-0.8,-0.25);
            \draw (-0.8,-0.25) to [out=0, in=45] (-0.4,0);
            \draw (-0.4,0) to [out=45, in=180] (0,0.25);
            \draw (0,0.25) to [out=0, in=135] (0.4,0);
            \draw (0.4,0) to [out=315, in=180] (0.8,-0.25);
            \draw (0.8,-0.25) to [out=0, in=45] (1.2,0);
            \draw (1.2,0) to [out=45, in=180] (1.6,0.25);
            \draw (1.6,0.25) to [out=0, in=135] (2,0);
        \end{scope}
    \node [below] at (0,-0.5) {$\mathcal{P}_{\rm fin}(\omega) \times \{ 0 \}$};
\end{scope}
\begin{scope}[shift={(3,0)}]
    \draw (-2,0) arc [radius=2, start angle=180, end angle=360];
    \begin{scope}[yscale=0.7]
           \draw (-2,0) to [out=45, in=180] (-1.6,0.25);
            \draw (-1.6,0.25) to [out=0, in=135] (-1.2,0);
            \draw (-1.2,0) to [out=315, in=180] (-0.8,-0.25);
            \draw (-0.8,-0.25) to [out=0, in=45] (-0.4,0);
            \draw (-0.4,0) to [out=45, in=180] (0,0.25);
            \draw (0,0.25) to [out=0, in=135] (0.4,0);
            \draw (0.4,0) to [out=315, in=180] (0.8,-0.25);
            \draw (0.8,-0.25) to [out=0, in=45] (1.2,0);
            \draw (1.2,0) to [out=45, in=180] (1.6,0.25);
            \draw (1.6,0.25) to [out=0, in=135] (2,0);
            \draw [fill=black] (-1.75,-0.8) circle (1.5pt);
            \node [below] at (-2.5,-0.5) {$(F,1)$};
        \end{scope}
    \node [below] at (0,-0.5) {$\mathcal{P}_{\rm fin}(\omega) \times \{ 1 \}$};
\end{scope}
\begin{scope}[shift={(3,2)}]
    \begin{scope}[yscale=-1]
        \draw (-2,0) arc [radius=2, start angle=180, end angle=360];
        \begin{scope}[yscale=0.7]
               \draw (-2,0) to [out=45, in=180] (-1.6,0.25);
                \draw (-1.6,0.25) to [out=0, in=135] (-1.2,0);
                \draw (-1.2,0) to [out=315, in=180] (-0.8,-0.25);
                \draw (-0.8,-0.25) to [out=0, in=45] (-0.4,0);
                \draw (-0.4,0) to [out=45, in=180] (0,0.25);
                \draw (0,0.25) to [out=0, in=135] (0.4,0);
                \draw (0.4,0) to [out=315, in=180] (0.8,-0.25);
                \draw (0.8,-0.25) to [out=0, in=45] (1.2,0);
                \draw (1.2,0) to [out=45, in=180] (1.6,0.25);
                \draw (1.6,0.25) to [out=0, in=135] (2,0);
            \end{scope}
        \node [below] at (0,-1.2) {$\mathcal{P}_{\rm cofin}(\omega) \times \{ 1 \}$};
    \end{scope}
\end{scope}
\draw (-4.8,-2.36) -- (1,0);
\draw (-2.12,-4.3) -- (3.83,-1.82);
\draw (1,0) -- (1,2);
\draw (5,0) -- (5,2);
\end{tikzpicture}
\end{center}
\caption{The lattice $L_3$.}
\label{fig:L3}
\end{figure}
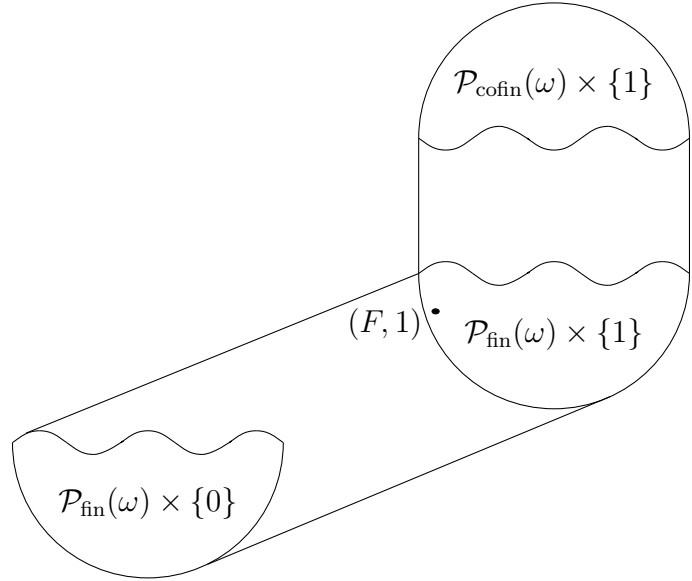

\begin{definition}
Let $L\in{\sf Dist}$ and let $a,b \in L$. Define
\[
I_{a \to b} := \{ c \in L \mid c \meet a \leq b \}
\]
\end{definition}

It is easy to check that $I_{a \to b}$ is an ideal, and that $I_{a \to b}$ is principal iff $a \impla b$ exists in $L$,
in which case $I_{a \to b} = {\downarrow} (a \impla b)$.

In order to give the dual description of $I_{a \to b}$ 
let $X$ be the Priestley space of $L$ and let $\alpha$ be the isomorphism from $L$ onto the lattice of clopen upsets of $X$ 
(see the introduction). It is well known that 
ideals of $L$ correspond to open upsets of $X$, and this correspondence is realized by sending an ideal $I$ of $L$ to
$\alpha[I]:= \bigcup \{ \alpha(a) \mid a \in I \}$. 
On the other hand,
filters of $L$ correspond to closed upsets of $X$, and this correspondence
is realized by sending a filter $F$ of $L$ to $\alpha[F]:= \bigcap \{ \alpha(a) \mid a \in F \}$.

\begin{lemma} \label{lemma:iatobcorrespond}
Let $L\in{\sf Dist}$ and let $X$ be its dual Priestley space.
If $a,b \in L$, then $\alpha[I_{a \to b}]=X \setminus \darr (\alpha(a) \setminus \alpha(b))$.
\end{lemma}

\begin{proof}
For any $c \in L$ we have
\begin{equation*}
\begin{split}
c \in I_{a \to b} \Leftrightarrow c \meet a \leq b & \Leftrightarrow \alpha(c) \cap \alpha(a) \subseteq \alpha(b) \\
& \Leftrightarrow \alpha(c) \subseteq X \setminus \darr (\alpha(a) \setminus \alpha(b))
\end{split}
\end{equation*}
where the last equivalence follows from the fact that for any upsets $U,V,W$ we have $W \cap U \subseteq V$ iff
$W \subseteq X \setminus {\downarrow} (U \setminus V)$. Thus, $\alpha[I_{a \to b}]=X \setminus \darr (\alpha(a) \setminus \alpha(b))$.
\end{proof}

It is a well-known consequence of Stone duality for Boolean algebras that a Boolean algebra is countable iff its Stone space is
metrizable (see, e.g., \cite[Prop.~7.23]{kop89}). This fact generalizes to bounded distributive lattices and Priestley spaces
(see, e.g., \cite[p.~54]{pri84}). To see this, let $L$ be a bounded distributive lattice and $X$ its Priestley space. The Boolean
algebra of clopens of $X$ is isomorphic to
the free Boolean extension $B(L)$ of $L$; see, e.g., \cite[Sec.~V.4]{bd11}. Thus, the following three conditions are equivalent:
\begin{itemize}
\item $X$ is metrizable;
\item $L$ is countable;
\item $B(L)$ is countable.
\end{itemize}

\begin{theorem}\label{cHA}
Let $L$ be a countable bounded distributive lattice. Then $L$ is not a Heyting algebra iff one of $L_i$ $(i=1,2,3)$ is a homomorphic
image of $L$ such that the homomorphism $h_i \colon L \to L_i$ satisfies the following property: There are $a,b \in L$ such that
$h_i[I_{a \to b}]=I_{c_i \to 0}$, where $c_1=c$, $c_2=\{ 0 \}$, or $c_3= (\varnothing,1)$.
\end{theorem}

\begin{proof}
$(\Rightarrow)$ It is sufficient to translate Theorem~\ref{thm:maintheorem} to its dual algebraic statement.
Let $X$ be the Priestley space of $L$. Then $X$ is a metrizable Priestley space which is not an Esakia space. Thus, by
Theorem~\ref{thm:maintheorem}, $Z_i$ is a forbidden configuration for $X$ for some $i=1,2,3$. Let $e,U$ be as in
Definition~\ref{def:forbidconf}. Then there are $a,b \in L$ such that $e(y) \in \alpha(a)\setminus \alpha(b) \subseteq U$.
Therefore, $e^{-1}{\downarrow}(\alpha(a) \setminus \alpha(b))\subseteq e^{-1}{\downarrow}U=\{x,y\}$. On the other hand, since
$e$ is order-preserving and $e(y)\in \alpha(a) \setminus \alpha(b)$, we have
$\{e(x),e(y)\}\subseteq {\downarrow}(\alpha(a) \setminus \alpha(b))$, so
$\{x,y\}\subseteq e^{-1}{\downarrow}(\alpha(a) \setminus \alpha(b))$. Thus,
$e^{-1}{\downarrow}(\alpha(a) \setminus \alpha(b))=\{x,y\}={\downarrow}y$.
We also have that $\alpha(c_i)= \{ y \} \subseteq Z_i$. By Lemma~\ref{lemma:iatobcorrespond},
$\alpha[I_{a \to b}]=X\setminus{\downarrow}(\alpha(a)\setminus \alpha(b))$ and $\alpha[I_{c_i \to 0}]=Z_i\setminus{\downarrow}y$.
Let $h_i \colon L\to L_i$ be the bounded lattice homomorphism corresponding to the embedding $e \colon Z_i\to X$, so $h_i=e^{-1}$.
Since $e$ is an embedding, $h_i$ is onto \cite{pri70}. Therefore, since
\[
e^{-1}(X\setminus{\downarrow}(\alpha(a) \setminus \alpha(b)))=Z_i\setminus e^{-1}{\downarrow}(\alpha(a) \setminus \alpha(b))
=Z_i\setminus{\downarrow}y,
\]
we conclude that $h_i[I_{a \to b}]=I_{c_i \to 0}$.

$(\Leftarrow)$
We show that $a \impla b$ does not exist in $L$. If $a \impla b$ exists, then we have $I_{a \to b} = {\downarrow} (a \impla b)$.
Since $h_i$ is an onto lattice homomorphism,
\[
I_{c_i \to 0}=h_i[I_{a \to b}]=h_i[{\downarrow} (a \impla b)]={\downarrow} h_i(a \impla b).
\]
Therefore, $c_i \impla 0=h_i(a \impla b)$, and hence $c_i \impla 0$ exists in $L_i$. The obtained contradiction proves that
$a \impla b$ does not exist in $L$, and hence $L$ is not a Heyting algebra.
\end{proof}

Theorem~\ref{cHA} yields a characterization of countable Heyting algebras. We conclude this section by showing that this
characterization easily generalizes to countable p-algebras. We recall (see, e.g., \cite{lee70}) that a \emph{p-algebra}
is a pseudocomplemented distributive lattice. Priestley duality for p-algebras was developed in \cite{pri75}. We call
a Priestley space $X$ a \emph{p-space} provided the downset of each open upset is open. Then a bounded distributive
lattice $L$ is a p-algebra iff its dual Priestley space $X$ is a p-space.

\begin{definition}
Let $X$ be a Priestley space. We say that $Z_i$ ($i=1,2,3$) is a \emph{p-configuration} for $X$ if $Z_i$ is a forbidden
configuration for $X$ and in addition the open neighborhood $U$ of $e(y)$ is an upset.
\end{definition}

We point out that neither of the bounded distributive lattices $L_1$, $L_2$, $L_3$ that are dual to $Z_1$, $Z_2$, $Z_3$
is a p-algebra. The next result is a direct generalization of Theorems~\ref{thm:maintheorem} and~\ref{cHA}, so we skip
its proof.

\begin{corollary} \label{thm:corollpseudocompl}
Let $L$ be a countable bounded distributive lattice, and let $X$ be its Priestley space, which is then a metrizable space.
\begin{enumerate}
\item $X$ is not a p-space iff one of $Z_1, Z_2, Z_3$ is a p-configuration for $X$.
\item $L$ is not a p-algebra iff one of $L_i$ $(i=1,2,3)$ is a homomorphic image of $L$ such that the homomorphism
$h_i \colon L \to L_i$ satisfies the following property: There is $a\in L$ such that $h_i[I_{a \to 0}]=I_{c_i \to 0}$, where
$c_1=c$, $c_2=\{ 0 \}$, or $c_3 = (\varnothing,1)$.
\end{enumerate}
\end{corollary}

We recall that \emph{co-Heyting algebras} are order-duals of Heyting algebras.
The Priestley spaces dual to co-Heyting algebras are the ones with the property that the upset of each open is open
\cite{esa75}. Let $Z_1^*,Z_2^*,Z_3^*$ be the Priestley spaces obtained by reversing the order in $Z_1,Z_2,Z_3$, respectively.
Then dualizing Theorem~\ref{thm:maintheorem} yields:

\begin{corollary}
A metrizable Priestley space $X$ is not the dual of a co-Heyting algebra iff there are a topological and order embedding $e$ from
one of $Z_1^*, Z_2^*, Z_3^*$ into $X$ and an open neighborhood $U$ of $e(y)$ such that $e^{-1}({\uparrow}U)=\{x,y\}$.
\end{corollary}

We recall that \emph{bi-Heyting algebras} are the lattices which are both Heyting algebras and co-Heyting algebras. Priestley spaces
dual to bi-Heyting algebras are the ones in which the upset and downset of each open is open. Putting together the results for Heyting
algebras and co-Heyting algebras yields:

\begin{corollary}
A metrizable Priestley space $X$ is not dual to a bi-Heyting algebra iff one of $Z_1, Z_2, Z_3$ is a forbidden configuration for $X$ or
there are a topological and order embedding $e$ from one of $Z_1^*, Z_2^*, Z_3^*$ into $X$ and an open neighborhood $U$ of $e(y)$ such that
$e^{-1}({\uparrow}U)=\{x,y\}$.
\end{corollary}

\section{Ideal-filter product and down-up sum}

\begin{definition}
Let $L$ and $M$ be bounded distributive lattices, $I$ an ideal of $L$, and $F$ a filter of $M$. We define the
\emph{ideal-filter product} of $L$ and $M$ as
\begin{equation*}
\lifm := \{ (l,m) \in L \times M \mid l \in I \text{ or } m \in F \}.
\end{equation*}
\end{definition}

\begin{lemma}
$\lifm$ is a bounded sublattice of $L \times M$.
\end{lemma}

\begin{proof}
Clearly $(0,0) \in \lifm$ since $0 \in I$, and $(1,1) \in \lifm$ because $1 \in F$. Let $(l_1,m_1), (l_2,m_2) \in \lifm$.
If $m_1 \meet m_2 \notin F$, then
$m_1 \notin F$ or $m_2 \notin F$, so $l_1 \in I$ or $l_2 \in I$, implying that $l_1 \meet l_2 \in I$. Therefore,
$(l_1 \meet l_2,m_1 \meet m_2) \in \lifm$. If $l_1 \join l_2 \notin I$, then
$l_1 \notin I$ or $l_2 \notin I$, so $m_1 \in F$ or $m_2 \in F$, which implies that $m_1 \join m_2 \in F$. Thus,
$(l_1 \join l_2,m_1 \join m_2) \in \lifm$.
\end{proof}

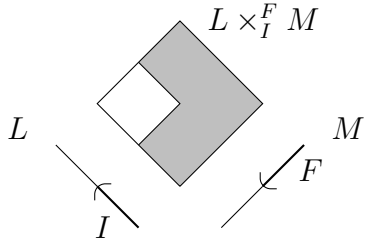
\begin{figure}[!h]
\centering
\begin{tikzpicture}[scale=1.1]
\node [right] at (0.2,2) {$\lifm$};
\path [fill=lightgray] (-0.5,0.5) -- (0,1) -- (-0.5,1.5) -- (0,2) -- (1,1) -- (0,0);
\draw (0,0) -- (1,1) -- (0,2) -- (-1,1) -- (0,0);
\draw (-0.5,0.5) -- (0,1) -- (-0.5,1.5);
\draw (0.5,-0.5) -- (1,0);
\draw [thick] (1,0) -- (1.5,0.5);
\draw [rounded corners] (0.97,0.17) -- (0.97,-0.03) -- (1.17,-0.03);
\draw [thick] (-0.5,-0.5) -- (-1,0);
\draw (-1,0) -- (-1.5,0.5);
\draw [rounded corners] (-1.03,-0.17) -- (-1.03,0.03) -- (-0.83,0.03);
\node [left] at (-1.7,0.7) {$L$};
\node [left] at (-0.7,-0.5) {$I$};
\node [right] at (1.7,0.7) {$M$};
\node [right] at (1.3,0.2) {$F$};
\end{tikzpicture}
\caption{$\lifm$ as a sublattice of $L \times M$.}
\label{fig:lifm}
\end{figure}

In order to describe the Priestley space of $\lifm$, we recall (see, e.g., \cite[p.~17 and p.~269]{dp02}) the definition of linear sum
of two Priestley spaces.
Let $X,Y$ be Priestley spaces. For simplicity, we assume for the rest of this section that $X$ and $Y$ are disjoint. If they are not,
then as usual, we can simply replace $X$ with $X \times \{ 0 \}$ and $Y$ with $Y \times \{ 1\}$. Define the \emph{linear sum} $X\oplus Y$
as the disjoint union of $X$ and $Y$ with the topology of disjoint union and the order given by

\begin{eqnarray*}
x \leq y & \mbox{ iff } & (x,y \in X \mbox{ and } x \leq y) \mbox{ or } \\
& & (x,y \in Y \mbox{ and } x \leq y)\mbox{ or } \\
& & (x \in X \mbox{ and } y \in Y).
\end{eqnarray*}

Intuitively, we
place $X$ ``below" $Y$. We next modify the definition of the linear sum of $X$ and $Y$.

\begin{definition}
Let $X,Y$ be Priestley spaces, $D$ a closed downset of $X$, and $U$ a closed upset of $Y$. We define the \emph{down-up sum}
$\xduy$ of $X$ and $Y$ as their disjoint union with the topology of disjoint union and the order given by
\begin{eqnarray*}
x \leq y & \mbox{ iff } & (x,y \in X \mbox{ and } x \leq y) \mbox{ or } \\
& & (x,y \in Y \mbox{ and } x \leq y)\mbox{ or } \\
& & (x \in D \mbox{ and } y \in U).
\end{eqnarray*}
\end{definition}

Intuitively, instead of
placing $X$ ``below" $Y$, we are only
placing $D$ ``below" $U$ (see Figure~\ref{fig:xduy}).

\begin{figure}[!h]
\centering
\begin{tikzpicture}[scale=0.7]
\draw (5,3) -- (-0.5,3) -- (-0.5,1) -- (5,1) -- (5,3);
\draw (-5,-3) -- (0.5,-3) -- (0.5,-1) -- (-5,-1) -- (-5,-3);
\draw [dotted] (0.5,-1) -- (1.5,1);
\draw (1.5,1) -- (2.5,3);
\draw [dotted] (-0.5,1) -- (-1.5,-1);
\draw (-1.5,-1) -- (-2.5,-3);
\draw [dotted] (1.5,1) -- (-1.5,-1);
\draw [dotted] (-0.5,1) -- (0.5,-1);
\node at (-0.5,-2) {$D$};
\node at (0.5,2) {$U$};
\node at (-4,-0.7) {$X$};
\node at (4,0.7) {$Y$};
\end{tikzpicture}
\caption{The Priestley space $\xduy$.}
\label{fig:xduy}
\end{figure}
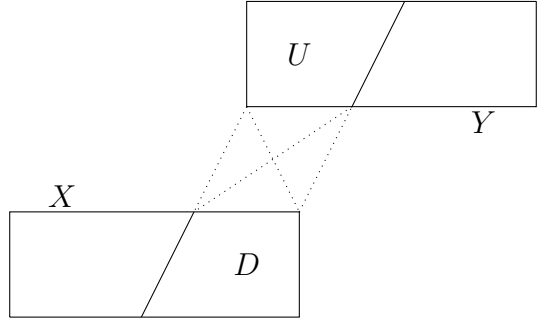

\begin{lemma}
$\xduy$ is a Priestley space.
\end{lemma}

\begin{proof}
Clearly $\xduy$ is compact. That $\leq$ is reflexive and antisymmetric is obvious, and that $\le$ is transitive follows from $D$ being
a downset of $X$ and $U$ an upset of $Y$. Let $x \nleq y$. First suppose that $x,y\in X$. Then there is a clopen upset $A$ of $X$ containing
$x$ and missing $y$. Therefore, $A \cup Y$ is a clopen upset of $\xduy$ containing $x$ and missing $y$.

Next suppose that $x,y \in Y$. Then there is a clopen upset $B$ of $Y$ containing $x$ and missing $y$. But then $B$
is also a clopen upset of $\xduy$ containing $x$ and missing $y$.

If $x \in Y$ and $y\in X$, then $Y$ is a clopen upset of $\xduy$ containing $x$ and missing $y$.
Finally, suppose that $x \in X$ and $y \in Y$. Since $x \nleq y$, we have $x \notin D$ or $y \notin U$.
If $x \notin D$, then since $D$ is a closed downset of $X$, there is a clopen upset $A$ of $X$ containing $x$ and disjoint from $D$.
Thus, $A$ is a clopen upset of $\xduy$ containing $x$ and missing $y$. If $y \notin U$, then since $U$ is a closed
upset of $Y$, there is a clopen downset $B$ of $Y$ containing $y$ and disjoint from $U$. Therefore, $A:=Y\setminus B$ is a clopen upset of $Y$
containing $U$ and missing $y$. Thus, $X \cup A$ is a clopen upset of $\xduy$ containing $x$ and missing $y$.
\end{proof}

\begin{theorem}
Let $L,M$ be bounded distributive lattices, $I$ an ideal of $L$, and $F$ a filter of $M$. Let also $X$ be the Priestley space of $L$,
$Y$ the Priestley space of $M$, $V$ an open upset of $X$ corresponding to the ideal $I$, $D:=X\setminus V$, and $U$ the closed upset
of $Y$ corresponding to the filter $F$. Then $\xduy$ is homeomorphic and order-isomorphic to the Priestley space of $\lifm$.
\end{theorem}

\begin{proof}
Let $\alpha$ be a lattice isomorphism from $L$ onto the clopen upsets of $X$ and $\beta$ a lattice isomorphism from $M$ onto the clopen
upsets of $Y$. Define $\gamma$ from $\lifm$ to the clopen upsets of $\xduy$ by $\gamma(l,m)=\alpha(l)\cup\beta(m)$. Since $l\in I$,
we have $\alpha(l)\cap D=\varnothing$; and since $m\in F$, we have $U\subseteq\beta(m)$. Thus, $\gamma(l,m)$ is a clopen upset of
$\xduy$, and so $\gamma$ is well defined. It is straightforward to see that $\gamma$ is a one-to-one bounded lattice homomorphism.
To see that $\gamma$ is onto, let $A$ be a clopen upset of $\xduy$. Let $l\in L$ and $m\in M$ be such that $\alpha(l)=A\cap X$ and
$\beta(m)=A\cap Y$.
If $l\notin I$, then $\alpha(l) \cap D \neq \varnothing$. So there is $d \in D\cap A$, and since $A$ is an upset of $\xduy$, we have
${\uparrow} d \subseteq A$. But then, by the definition of the order on $\xduy$, we have that $U \subseteq {\uparrow} d \subseteq A$.
Therefore, $U \subseteq A \cap Y=\beta(m)$, and so $m\in F$. Thus, $(l,m)\in\lifm$, and hence $\gamma$ is onto. Consequently, $\lifm$
is isomorphic to the clopen upsets of $\xduy$, which by Priestley duality yields that $\xduy$ is homeomorphic and order-isomorphic to
the Priestley space of $\lifm$.
\end{proof}

\section{Counterexamples} \label{section:counterexamples}

In the definition of $\xduy$, when the closed upset $U$ coincides with $Y$, we denote $\xduy$ by $\xdy$.

\begin{lemma} \label{lem:D clopen}
Let $X,Y$ be Esakia spaces. Then $\xdy$ is an Esakia space iff $D$ is clopen in $X$.
\end{lemma}

\begin{proof}
Without loss of generality we may assume that $X$ and $Y$ are disjoint. First suppose that $D$ is not clopen in $X$.
We have that $Y$ is a clopen upset of $\xdy$ and ${\downarrow} Y= Y \cup D$. Since $(Y \cup D) \cap X= D$, we see
that $Y \cup D$ cannot be clopen in $\xdy$. Therefore, $\xdy$ is not an Esakia space.

Next suppose that $D$ is clopen in $X$. Any clopen in $\xdy$ can be written as $A \cup B$ where $A$ is clopen in $X$ and $B$ is clopen in $Y$.
If $B = \varnothing$, then ${\downarrow} (A \cup B) = {\downarrow} A$ is clopen in $X$, and hence clopen in $\xdy$. If $B \neq \varnothing$, then
${\downarrow} (A \cup B) = {\downarrow} A \cup {\downarrow} B = {\downarrow} A \cup ({\downarrow} B \cap Y) \cup D$. Since ${\downarrow} A,D$
are clopen in $X$ and ${\downarrow} B \cap Y$ is clopen in $Y$, we conclude that ${\downarrow} (A \cup B)$ is clopen in $\xdy$.
\end{proof}

\begin{remark}
\begin{enumerate}
\item[]
\item There is an obvious analogue of Lemma~\ref{lem:D clopen} for p-spaces: For p-spaces $X$ and $Y$, $\xdy$ is a p-space iff
$D$ is clopen in $X$.
\item If $Y= \{ y \}$ is a singleton, then in the definition of $\xdy$ we are adding only one point on top of $D$.
\end{enumerate}
\end{remark}

We are ready to give examples of non-metrizable (even non-sequential) Priestley spaces such that they are not Esakia spaces and yet
they do not contain the three forbidden configurations.

\begin{example} \label{ex:stonecech}
Let $X=\beta \omega$ be the Stone-\v{C}ech compactification of the discrete space $\omega$. We view $\beta \omega$
as an Esakia space with trivial order. Let $D= \beta \omega \setminus \omega$,
$Y= \{ y \}$, and consider $\xdy$; see Figure~\ref{fig:stonecech}.

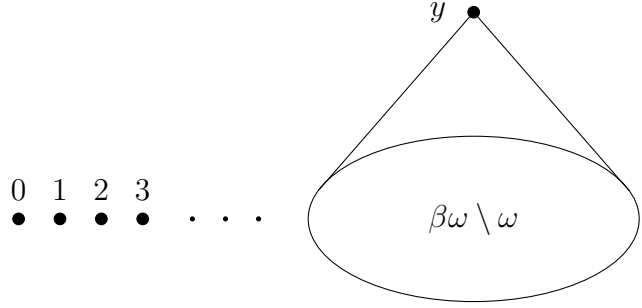
\begin{figure}[!h]
\begin{center}
\begin{tikzpicture}[scale=1.1]
\draw (2,0) ellipse (2cm and 1cm);
\draw [fill=black] (-2,0) circle (2.0pt);
\draw [fill=black] (-2.5,0) circle (2.0pt);
\draw [fill=black] (-3,0) circle (2.0pt);
\draw [fill=black] (-3.5,0) circle (2.0pt);
\draw [fill=black] (-0.6,0) circle (0.7pt);
\draw [fill=black] (-1.0,0) circle (0.7pt);
\draw [fill=black] (-1.4,0) circle (0.7pt);
\node [above] at (-3.5,0.1) {$0$};
\node [above] at (-3,0.1) {$1$};
\node [above] at (-2.5,0.1) {$2$};
\node [above] at (-2,0.1) {$3$};
\draw [fill=black] (2,2.5) circle (2.0pt);
\node [left] at (1.8,2.5) {$y$};
\draw (2,2.5) -- (3.88,0.34);
\draw (2,2.5) -- (0.12,0.34);
\node  at (2,0) {$\beta \omega \setminus \omega$};
\end{tikzpicture}
\end{center}
\caption{The space $X \oplus_D Y$ of Example~\ref{ex:stonecech}.}
\label{fig:stonecech}
\end{figure}

Since $D$ is not clopen, Lemma~\ref{lem:D clopen} implies that this is an example of a Priestley space that is not an Esakia space.
It is well known (see, e.g., \cite[Cor.~3.6.15]{eng89}) that there are no non-trivial convergent sequences in $\beta\omega$.
Therefore, there is no sequence in $({\downarrow}y)^c$ converging to an element of ${\downarrow} y$. Thus, $\xdy$
does not contain the three forbidden configurations.

The closed downset $D$ corresponds to the ideal $I:=\mathcal{P}_{\rm fin}(\omega)$ of finite subsets of $\mathcal{P}(\omega)$,
and the clopen upset $Y=\{ y \}$ corresponds to the filter $F=\{ 1 \}$ of $\bf 2$. Thus, the dual lattice of $\xdy$
is $\mathcal{P}(\omega) \times_I^F {\bf 2}$; see Figure \ref{fig:dualstonecech}.

\begin{figure}[!h]
\begin{center}
\hspace*{-0.7cm}
\begin{tikzpicture}[scale=0.8]
\begin{scope}[shift={(-3,-2.5)}]
    \draw (-2,0) arc [radius=2, start angle=180, end angle=360];
    \begin{scope}[yscale=0.7]
           \draw (-2,0) to [out=45, in=180] (-1.6,0.25);
            \draw (-1.6,0.25) to [out=0, in=135] (-1.2,0);
            \draw (-1.2,0) to [out=315, in=180] (-0.8,-0.25);
            \draw (-0.8,-0.25) to [out=0, in=45] (-0.4,0);
            \draw (-0.4,0) to [out=45, in=180] (0,0.25);
            \draw (0,0.25) to [out=0, in=135] (0.4,0);
            \draw (0.4,0) to [out=315, in=180] (0.8,-0.25);
            \draw (0.8,-0.25) to [out=0, in=45] (1.2,0);
            \draw (1.2,0) to [out=45, in=180] (1.6,0.25);
            \draw (1.6,0.25) to [out=0, in=135] (2,0);
        \end{scope}
    \node [below] at (0,-0.5) {$\mathcal{P}_{\rm fin}(\omega) \times \{ 0 \}$};
\end{scope}
\begin{scope}[shift={(3,0)}]
    \draw (-2,0) arc [radius=2, start angle=180, end angle=360];
    \node [below] at (0,1) {$\mathcal{P}(\omega) \times \{ 1 \}$};
\end{scope}
\begin{scope}[shift={(3,1)}]
    \begin{scope}[yscale=-1]
        \draw (-2,0) arc [radius=2, start angle=180, end angle=360];
    \end{scope}
\end{scope}
\draw (-4.8,-2.36) -- (1,0);
\draw (-2.12,-4.3) -- (3.83,-1.82);
\draw (1,0) -- (1,1);
\draw (5,0) -- (5,1);
\end{tikzpicture}
\end{center}
\caption{The dual lattice of the space $X \oplus_D Y$ of Example~\ref{ex:stonecech}.}
\label{fig:dualstonecech}
\end{figure}
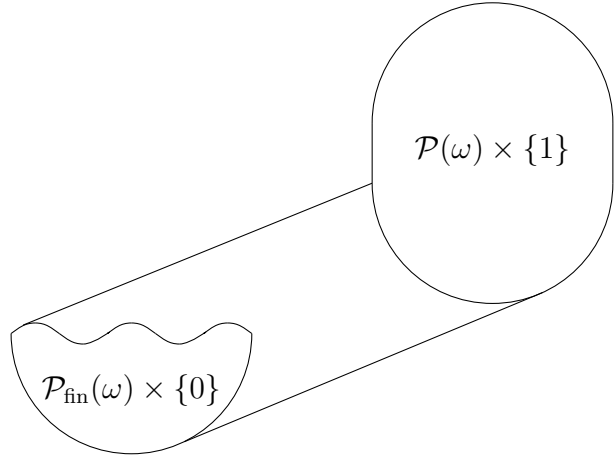
\end{example}

\begin{example}\label{ex:omega1plus1}
Let $\omega_1$ be the first uncountable ordinal, and let $X$ be the poset obtained by taking the dual order of $\omega_1+1$.
Endow $X$ with the interval topology. It is straightforward to check that $X$ is an Esakia space.
Let $D=\{ \omega_1 \} \subseteq X$,
let $Y=\{ y \}$, and consider $\xdy$; see Figure~\ref{fig:omega1plus1}.

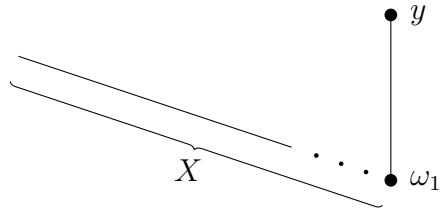
\begin{figure}[!h]
\begin{center}
\begin{tikzpicture}[scale=1.1]
\draw [fill=black] (0,0) circle (2.0pt);
\draw [fill=black] (0,2) circle (2.0pt);
\draw (0,0) -- (0,2);
\draw (-1.2,0.4) -- (-4.5,1.5);
\node [right] at (0.1,0) {$\omega_1$};
\node [right] at (0.1,2) {$y$};
\draw[decoration={brace,mirror,raise=10pt},decorate] (-4.5,1.5) -- (0,0);
\draw [fill=black] (-0.3,0.1) circle (0.6pt);
\draw [fill=black] (-0.6,0.2) circle (0.6pt);
\draw [fill=black] (-0.9,0.3) circle (0.6pt);
\node [right] at (-2.75,0.12) {$X$};
\end{tikzpicture}
\end{center}
\caption{The space $X \oplus_D Y$ of Example~\ref{ex:omega1plus1}.}
\label{fig:omega1plus1}
\end{figure}

Since $D$ is not clopen, Lemma~\ref{lem:D clopen} implies that $\xdy$ is not an Esakia space.
On the other hand,
there is no sequence in $X \setminus \{ \omega_1 \}$ converging to $ \omega_1$. Thus, $\xdy$
does not contain the three forbidden configurations.

The dual lattice of
$X$ is $\omega_1+1$, the closed downset $D$ corresponds to the ideal $I:=\omega_1$
of $\omega_1+1$, and the clopen upset $Y=\{ y \}$ corresponds to the filter $F=\{ 1 \}$ of $\bf 2$. Thus, the dual lattice of
$\xdy$ is $(\omega_1+1) \times_I^F {\bf 2}$; see Figure \ref{fig:uncountableladder}.
\end{example}

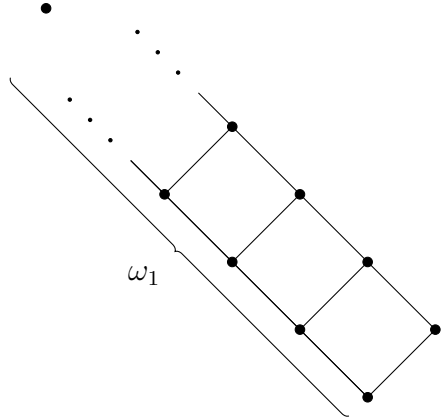
\begin{figure}[!h]
\begin{center}
\begin{tikzpicture}[scale=0.9]
\draw [fill=black] (0,0) circle (2.0pt);
\draw [fill=black] (-1,1) circle (2.0pt);
\draw [fill=black] (0,2) circle (2.0pt);
\draw [fill=black] (1,1) circle (2.0pt);
\draw [fill=black] (-1,3) circle (2.0pt);
\draw [fill=black] (-2,2) circle (2.0pt);
\draw [fill=black] (-2,4) circle (2.0pt);
\draw [fill=black] (-3,3) circle (2.0pt);
\draw [fill=black] (-4.75,5.75) circle (2.0pt);
\draw (0,0) -- (-1,1) -- (-2,2) -- (-3,3) -- (-3.5,3.5) ;
\draw (1,1) -- (0,2) -- (-1,3) -- (-2,4) -- (-2.5,4.5) ;
\draw (0,0) -- (-1,1) -- (-2,2) -- (-3,3) -- (-3.5,3.5) ;
\draw (0,0) -- (1,1);
\draw (-1,1) -- (0,2);
\draw (-2,2) -- (-1,3);
\draw (-3,3) -- (-2,4);
\draw [fill=black] (-3.8,3.8) circle (0.7pt);
\draw [fill=black] (-4.1,4.1) circle (0.7pt);
\draw [fill=black] (-4.4,4.4) circle (0.7pt);
\draw [fill=black] (-2.8,4.8) circle (0.7pt);
\draw [fill=black] (-3.1,5.1) circle (0.7pt);
\draw [fill=black] (-3.4,5.4) circle (0.7pt);
\draw[decoration={brace,raise=10pt},decorate] (0,0) -- (-5,5);
\node at (-3.3,1.8) {$\omega_1$};
\end{tikzpicture}
\end{center}
\caption{The dual lattice of the space $X \oplus_D Y$ of Example~\ref{ex:omega1plus1}.}
\label{fig:uncountableladder}
\end{figure}

\begin{remark}
\begin{enumerate}
\item[]
\item The space $\xdy$ of Example~\ref{ex:omega1plus1} can be thought of as a generalization of the forbidden configuration $Z_1$,
obtained by ``stretching" the chain $\{z_n\mid n\in \omega\}$. As a result, the chain $X \setminus \{ \omega_1 \} $ is ``too long"
to contain a sequence converging to $\omega_1$.
\item We can generalize the forbidden configuration $Z_2$ similarly by ``stretching" the antichain $\{z_n\mid n\in \omega\}$.
\item The space $\xdy$ of Example~\ref{ex:stonecech} can be thought of as a generalization of the forbidden configuration $Z_3$,
obtained by ``inflating" the point $x$. As a result,
we do not have sequences from $\omega$ converging inside $\beta\omega\setminus\omega$.
\end{enumerate}
\end{remark}

We note that in the definition of the three forbidden configurations, the condition on the open neighborhood $U$ of $e(y)$
cannot be dropped. This can be seen from a general observation that every Priestley space embeds into an Esakia space, and
hence every bounded distributive lattice is a homomorphic image of a Heyting algebra. To see this, let $L$ be a bounded
distributive lattice and $X$ its Priestley space. We let $F_L$ be the free bounded distributive lattice generated by the
underlying set of $L$. The identity on $L$ induces an onto lattice homomorphism $h \colon F_L \to L$. Dually, the onto homomorphism $h$
corresponds to an embedding $e \colon X \to 2^L$ where $2=\{0,1\}$ is the two-element discrete Priestley space with $0<1$.
Since $2$ is an Esakia space and products of Esakia spaces are Esakia spaces,
$2^L$ is an Esakia space. Thus, $F_L$ is a Heyting algebra.
Consequently, we cannot characterize Esakia spaces by forbidding embeddings of some Priestley spaces. This yields
that in the definition of the three forbidden configurations, the condition on the open neighborhood $U$ of $e(y)$
cannot be dropped.

In most cases, the space $2^L$
is rather complex.
We conclude the paper by presenting much simpler examples of Esakia spaces into which the Priestley spaces $Z_1, Z_2$ and $Z_3$ embed.

\begin{example} \label{ex:Z1embedcounterex}
Let $X$ be the disjoint union of two copies of the one-point compactification of the discrete space $\omega$, and let the order
on $X$ be defined as in Figure~\ref{fig:Z1embedcounterex}.

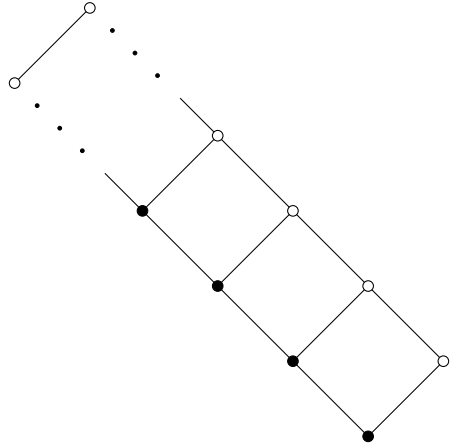
\begin{figure}[!h]
\begin{center}
\begin{tikzpicture}
\draw (0,0) -- (-1,1) -- (-2,2) -- (-3,3) -- (-3.5,3.5) ;
\draw (1,1) -- (-0,2) -- (-1,3) -- (-2,4) -- (-2.5,4.5) ;
\draw (0,0) -- (1,1);
\draw (-1,1) -- (0,2);
\draw (-2,2) -- (-1,3);
\draw (-3,3) -- (-2,4);
\draw (-4.7,4.7) -- (-3.7,5.7);
\draw [fill=black] (0,0) circle (2.0pt);
\draw [fill=black] (-1,1) circle (2.0pt);
\draw [fill=white] (0,2) circle (2.0pt);
\draw [fill=white] (1,1) circle (2.0pt);
\draw [fill=white] (-1,3) circle (2.0pt);
\draw [fill=black] (-2,2) circle (2.0pt);
\draw [fill=white] (-2,4) circle (2.0pt);
\draw [fill=black] (-3,3) circle (2.0pt);
\draw [fill=white] (-4.7,4.7) circle (2.0pt);
\draw [fill=white] (-3.7,5.7) circle (2.0pt);
\draw [fill=black] (-3.8,3.8) circle (0.7pt);
\draw [fill=black] (-4.1,4.1) circle (0.7pt);
\draw [fill=black] (-4.4,4.4) circle (0.7pt);
\draw [fill=black] (-2.8,4.8) circle (0.7pt);
\draw [fill=black] (-3.1,5.1) circle (0.7pt);
\draw [fill=black] (-3.4,5.4) circle (0.7pt);
\end{tikzpicture}
\caption{The space $X$ of Example~\ref{ex:Z1embedcounterex}. The white dots represent the image of $Z_1$
under the embedding of $Z_1$ into $X$.}
\label{fig:Z1embedcounterex}
\end{center}
\end{figure}

It is straightforward to check that $X$ is a metrizable Esakia space, and yet there is a topological and order embedding of $Z_1$ into
$X$, described by the white dots in the figure.
\end{example}

An analogous space for $Z_3$ can be constructed as follows.

\begin{example} \label{ex:Z3embedcounterex}
Let $X$ be the disjoint union of two copies of the one-point compactification of the discrete space $\omega$, and let the order
on $X$ be defined as in Figure~\ref{fig:Z3embedcounterex}.
\begin{figure}[!h]
\begin{center}
\begin{tikzpicture}
\draw (0,0) -- (0,1.5);
\draw (-1.5,1.5) -- (-1.5,0);
\draw (-3,0) -- (-3,1.5);
\draw (-4.5,1.5) -- (-4.5,0);
\draw (2.5,0) -- (2.5,1.5);
\draw [fill=white] (0,0) circle (2.0pt);
\draw [fill=black] (0,1.5) circle (2.0pt);
\draw [fill=white] (-1.5,0) circle (2.0pt);
\draw [fill=black] (-1.5,1.5) circle (2.0pt);
\draw [fill=white] (-3,0) circle (2.0pt);
\draw [fill=black] (-3,1.5) circle (2.0pt);
\draw [fill=white] (-4.5,0) circle (2.0pt);
\draw [fill=black] (-4.5,1.5) circle (2.0pt);
\draw [fill=white] (2.5,0) circle (2.0pt);
\draw [fill=white] (2.5,1.5) circle (2.0pt);
\draw [fill=black] (0.85,0) circle (0.7pt);
\draw [fill=black] (1.25,0) circle (0.7pt);
\draw [fill=black] (1.65,0) circle (0.7pt);
\draw [fill=black] (0.85,1.5) circle (0.7pt);
\draw [fill=black] (1.25,1.5) circle (0.7pt);
\draw [fill=black] (1.65,1.5) circle (0.7pt);
\end{tikzpicture}
\caption{The space $X$ of Example~\ref{ex:Z3embedcounterex}. The white dots represent the image of $Z_3$
under the embedding of $Z_3$ into $X$.}
\label{fig:Z3embedcounterex}
\end{center}
\end{figure}
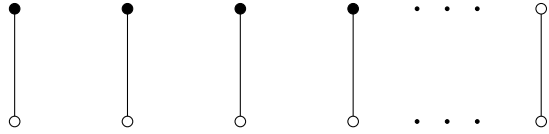

Then $X$ is a metrizable Esakia space, and the white dots describe a topological and order embedding of $Z_3$ into $X$.
\end{example}

We finish by constructing a metrizable Esakia space in which $Z_2$ is embedded, which is more complicated than
Examples~\ref{ex:Z1embedcounterex} and \ref{ex:Z3embedcounterex}.

\begin{example} \label{ex:Z2embedcounterex}
Let $X$ be a subspace of $\mathbb{R}^2$ as described in Figure~\ref{fig:Z2embedtopology} with each $\alpha_{mn}$ an isolated point,
each sequence $\{\alpha_{in}\mid n\in\omega\}$ converging to $\alpha_{i\omega}$, and each sequence $\{\alpha_{ni}\mid n\in\omega\}$
converging to $\alpha_{\omega i}$.

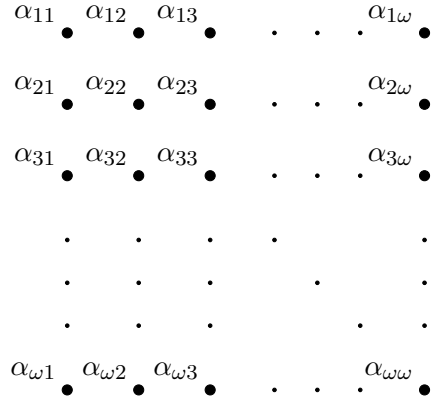
\begin{figure}[!h]
\begin{center}
\begin{tikzpicture}[scale=0.95]
\node [above left] at (0,0) {{\small $\alpha_{11}$}};
\node [above left] at (1,0) {{\small $\alpha_{12}$}};
\node [above left] at (2,0) {{\small $\alpha_{13}$}};
\node [above left] at (5,0) {{\small $\alpha_{1 \omega}$}};
\node [above left] at (0,-1) {{\small $\alpha_{21}$}};
\node [above left] at (1,-1) {{\small $\alpha_{22}$}};
\node [above left] at (2,-1) {{\small $\alpha_{23}$}};
\node [above left] at (5,-1) {{\small $\alpha_{2 \omega}$}};
\node [above left] at (0,-2) {{\small $\alpha_{31}$}};
\node [above left] at (1,-2) {{\small $\alpha_{32}$}};
\node [above left] at (2,-2) {{\small $\alpha_{33}$}};
\node [above left] at (5,-2) {{\small $\alpha_{3 \omega}$}};
\node [above left] at (0,-5) {{\small $\alpha_{\omega 1}$}};
\node [above left] at (1,-5) {{\small $\alpha_{\omega 2}$}};
\node [above left] at (2,-5) {{\small $\alpha_{\omega 3}$}};
\node [above left] at (5,-5) {{\small $\alpha_{\omega \omega}$}};
\draw [fill=black] (3.5,0) circle (0.7pt);
\draw [fill=black] (4.1,0) circle (0.7pt);
\draw [fill=black] (2.9,0) circle (0.7pt);
\draw [fill=black] (3.5,-1) circle (0.7pt);
\draw [fill=black] (4.1,-1) circle (0.7pt);
\draw [fill=black] (2.9,-1) circle (0.7pt);
\draw [fill=black] (3.5,-2) circle (0.7pt);
\draw [fill=black] (4.1,-2) circle (0.7pt);
\draw [fill=black] (2.9,-2) circle (0.7pt);
\draw [fill=black] (3.5,-5) circle (0.7pt);
\draw [fill=black] (4.1,-5) circle (0.7pt);
\draw [fill=black] (2.9,-5) circle (0.7pt);
\draw [fill=black] (0,-3.5) circle (0.7pt);
\draw [fill=black] (0,-4.1) circle (0.7pt);
\draw [fill=black] (0,-2.9) circle (0.7pt);
\draw [fill=black] (1,-3.5) circle (0.7pt);
\draw [fill=black] (1,-4.1) circle (0.7pt);
\draw [fill=black] (1,-2.9) circle (0.7pt);
\draw [fill=black] (2,-3.5) circle (0.7pt);
\draw [fill=black] (2,-4.1) circle (0.7pt);
\draw [fill=black] (2,-2.9) circle (0.7pt);
\draw [fill=black] (5,-3.5) circle (0.7pt);
\draw [fill=black] (5,-4.1) circle (0.7pt);
\draw [fill=black] (5,-2.9) circle (0.7pt);
\draw [fill=black] (3.5,-3.5) circle (0.7pt);
\draw [fill=black] (4.1,-4.1) circle (0.7pt);
\draw [fill=black] (2.9,-2.9) circle (0.7pt);
\draw [fill=black] (0,0) circle (2.0pt);
\draw [fill=black] (1,0) circle (2.0pt);
\draw [fill=black] (2,0) circle (2.0pt);
\draw [fill=black] (5,0) circle (2.0pt);
\draw [fill=black] (0,-1) circle (2.0pt);
\draw [fill=black] (1,-1) circle (2.0pt);
\draw [fill=black] (2,-1) circle (2.0pt);
\draw [fill=black] (5,-1) circle (2.0pt);
\draw [fill=black] (0,-2) circle (2.0pt);
\draw [fill=black] (1,-2) circle (2.0pt);
\draw [fill=black] (2,-2) circle (2.0pt);
\draw [fill=black] (5,-2) circle (2.0pt);
\draw [fill=black] (0,-5) circle (2.0pt);
\draw [fill=black] (1,-5) circle (2.0pt);
\draw [fill=black] (2,-5) circle (2.0pt);
\draw [fill=black] (5,-5) circle (2.0pt);
\end{tikzpicture}
\caption{The space $X$.
}
\label{fig:Z2embedtopology}
\end{center}
\end{figure}
Let $Y=\{\alpha_{mn}\mid m,n\in\omega\}$. Then $Y$ is a discrete subspace of $X$ and $X$ is a compactification of $Y$. Clearly $X$ is a
compact metrizable space.
Each clopen $U$ of $X$ is either a finite union of subsets of the form
\begin{itemize}
\item $\{ \alpha_{mn} \}$ for some $m,n \in \omega$;
\item $\{ \alpha_{m n} \mid h \le m \le \omega \}$ for some $h,n \in \omega$;
\item $\{ \alpha_{m n} \mid k \le n \le \omega \}$ for some $k,m \in \omega$;
\end{itemize}
or the complement of one of these finite unions.
From this it is easy to see that $X$ is a Stone space.

Define a partial order on $X$ by
\begin{equation*}
\alpha_{hk} \leq \alpha_{mn} \mbox{ iff } (h,k)=(m,n) \mbox{ or } h,k \geq m+n
\end{equation*}
where we set $m+n=\omega$ if at least one of $m,n$ is $\omega$.

\begin{figure}[!h]
\begin{center}
\begin{tikzpicture}[scale=0.75]
\path [fill=lightgray, rounded corners] (-0.7,-0.3) -- (-0.7,-1.7) -- (0.7,-1.7) -- (0.7,-0.3) -- cycle;
\path [fill=lightgray, rounded corners] (1.3,-1.3) -- (1.3,-5.7) -- (5.7,-5.7) -- (5.7,-1.3) -- cycle;
\node [above] at (0,-1) {$\alpha_{21}$};
\node [above] at (2,-2) {$\alpha_{33}$};
\node [above] at (5,-2) { $\alpha_{3 \omega}$};
\node [above] at (2,-5) { $\alpha_{\omega 3}$};
\node [above] at (5,-5) { $\alpha_{\omega \omega}$};
\draw [fill=black] (3.5,0) circle (0.7pt);
\draw [fill=black] (4.1,0) circle (0.7pt);
\draw [fill=black] (2.9,0) circle (0.7pt);
\draw [fill=black] (3.5,-1) circle (0.7pt);
\draw [fill=black] (4.1,-1) circle (0.7pt);
\draw [fill=black] (2.9,-1) circle (0.7pt);
\draw [fill=black] (3.5,-2) circle (0.7pt);
\draw [fill=black] (4.1,-2) circle (0.7pt);
\draw [fill=black] (2.9,-2) circle (0.7pt);
\draw [fill=black] (3.5,-5) circle (0.7pt);
\draw [fill=black] (4.1,-5) circle (0.7pt);
\draw [fill=black] (2.9,-5) circle (0.7pt);
\draw [fill=black] (0,-3.5) circle (0.7pt);
\draw [fill=black] (0,-4.1) circle (0.7pt);
\draw [fill=black] (0,-2.9) circle (0.7pt);
\draw [fill=black] (1,-3.5) circle (0.7pt);
\draw [fill=black] (1,-4.1) circle (0.7pt);
\draw [fill=black] (1,-2.9) circle (0.7pt);
\draw [fill=black] (2,-3.5) circle (0.7pt);
\draw [fill=black] (2,-4.1) circle (0.7pt);
\draw [fill=black] (2,-2.9) circle (0.7pt);
\draw [fill=black] (5,-3.5) circle (0.7pt);
\draw [fill=black] (5,-4.1) circle (0.7pt);
\draw [fill=black] (5,-2.9) circle (0.7pt);
\draw [fill=black] (3.5,-3.5) circle (0.7pt);
\draw [fill=black] (4.1,-4.1) circle (0.7pt);
\draw [fill=black] (2.9,-2.9) circle (0.7pt);
\draw [fill=black] (0,0) circle (2.0pt);
\draw [fill=black] (1,0) circle (2.0pt);
\draw [fill=black] (2,0) circle (2.0pt);
\draw [fill=black] (5,0) circle (2.0pt);
\draw [fill=black] (0,-1) circle (2.0pt);
\draw [fill=black] (1,-1) circle (2.0pt);
\draw [fill=black] (2,-1) circle (2.0pt);
\draw [fill=black] (5,-1) circle (2.0pt);
\draw [fill=black] (0,-2) circle (2.0pt);
\draw [fill=black] (1,-2) circle (2.0pt);
\draw [fill=black] (2,-2) circle (2.0pt);
\draw [fill=black] (5,-2) circle (2.0pt);
\draw [fill=black] (0,-5) circle (2.0pt);
\draw [fill=black] (1,-5) circle (2.0pt);
\draw [fill=black] (2,-5) circle (2.0pt);
\draw [fill=black] (5,-5) circle (2.0pt);
\node [below] at (3.2,-5.7) { $\darr \alpha_{21}$};
\end{tikzpicture}
\hspace{3.5cm}
\begin{tikzpicture}[scale=0.75]
\path [fill=lightgray, rounded corners] (1.3,-5.7) -- (2.7,-5.7) -- (2.7,-4.3) -- (1.3,-4.3) -- cycle;
\path [fill=lightgray, rounded corners] (4.3,-5.7) -- (5.7,-5.7) -- (5.7,-4.3) -- (4.3,-4.3) -- cycle;
\node [above] at (2,-5) { $\alpha_{\omega 3}$};
\node [above] at (5,-5) { $\alpha_{\omega \omega}$};
\draw [fill=black] (3.5,0) circle (0.7pt);
\draw [fill=black] (4.1,0) circle (0.7pt);
\draw [fill=black] (2.9,0) circle (0.7pt);
\draw [fill=black] (3.5,-1) circle (0.7pt);
\draw [fill=black] (4.1,-1) circle (0.7pt);
\draw [fill=black] (2.9,-1) circle (0.7pt);
\draw [fill=black] (3.5,-2) circle (0.7pt);
\draw [fill=black] (4.1,-2) circle (0.7pt);
\draw [fill=black] (2.9,-2) circle (0.7pt);
\draw [fill=black] (3.5,-5) circle (0.7pt);
\draw [fill=black] (4.1,-5) circle (0.7pt);
\draw [fill=black] (2.9,-5) circle (0.7pt);
\draw [fill=black] (0,-3.5) circle (0.7pt);
\draw [fill=black] (0,-4.1) circle (0.7pt);
\draw [fill=black] (0,-2.9) circle (0.7pt);
\draw [fill=black] (1,-3.5) circle (0.7pt);
\draw [fill=black] (1,-4.1) circle (0.7pt);
\draw [fill=black] (1,-2.9) circle (0.7pt);
\draw [fill=black] (2,-3.5) circle (0.7pt);
\draw [fill=black] (2,-4.1) circle (0.7pt);
\draw [fill=black] (2,-2.9) circle (0.7pt);
\draw [fill=black] (5,-3.5) circle (0.7pt);
\draw [fill=black] (5,-4.1) circle (0.7pt);
\draw [fill=black] (5,-2.9) circle (0.7pt);
\draw [fill=black] (3.5,-3.5) circle (0.7pt);
\draw [fill=black] (4.1,-4.1) circle (0.7pt);
\draw [fill=black] (2.9,-2.9) circle (0.7pt);
\draw [fill=black] (0,0) circle (2.0pt);
\draw [fill=black] (1,0) circle (2.0pt);
\draw [fill=black] (2,0) circle (2.0pt);
\draw [fill=black] (5,0) circle (2.0pt);
\draw [fill=black] (0,-1) circle (2.0pt);
\draw [fill=black] (1,-1) circle (2.0pt);
\draw [fill=black] (2,-1) circle (2.0pt);
\draw [fill=black] (5,-1) circle (2.0pt);
\draw [fill=black] (0,-2) circle (2.0pt);
\draw [fill=black] (1,-2) circle (2.0pt);
\draw [fill=black] (2,-2) circle (2.0pt);
\draw [fill=black] (5,-2) circle (2.0pt);
\draw [fill=black] (0,-5) circle (2.0pt);
\draw [fill=black] (1,-5) circle (2.0pt);
\draw [fill=black] (2,-5) circle (2.0pt);
\draw [fill=black] (5,-5) circle (2.0pt);
\node [below] at (3.2,-5.7) { $\darr \alpha_{\omega 3}$};
\end{tikzpicture}
\newline
\newline
\begin{tikzpicture}[scale=0.75]
\path [fill=lightgray, rounded corners] (4.3,-5.7) -- (5.7,-5.7) -- (5.7,-4.3) -- (4.3,-4.3) -- cycle;
\node [above] at (5,-5) { $\alpha_{\omega \omega}$};
\draw [fill=black] (3.5,0) circle (0.7pt);
\draw [fill=black] (4.1,0) circle (0.7pt);
\draw [fill=black] (2.9,0) circle (0.7pt);
\draw [fill=black] (3.5,-1) circle (0.7pt);
\draw [fill=black] (4.1,-1) circle (0.7pt);
\draw [fill=black] (2.9,-1) circle (0.7pt);
\draw [fill=black] (3.5,-2) circle (0.7pt);
\draw [fill=black] (4.1,-2) circle (0.7pt);
\draw [fill=black] (2.9,-2) circle (0.7pt);
\draw [fill=black] (3.5,-5) circle (0.7pt);
\draw [fill=black] (4.1,-5) circle (0.7pt);
\draw [fill=black] (2.9,-5) circle (0.7pt);
\draw [fill=black] (0,-3.5) circle (0.7pt);
\draw [fill=black] (0,-4.1) circle (0.7pt);
\draw [fill=black] (0,-2.9) circle (0.7pt);
\draw [fill=black] (1,-3.5) circle (0.7pt);
\draw [fill=black] (1,-4.1) circle (0.7pt);
\draw [fill=black] (1,-2.9) circle (0.7pt);
\draw [fill=black] (2,-3.5) circle (0.7pt);
\draw [fill=black] (2,-4.1) circle (0.7pt);
\draw [fill=black] (2,-2.9) circle (0.7pt);
\draw [fill=black] (5,-3.5) circle (0.7pt);
\draw [fill=black] (5,-4.1) circle (0.7pt);
\draw [fill=black] (5,-2.9) circle (0.7pt);
\draw [fill=black] (3.5,-3.5) circle (0.7pt);
\draw [fill=black] (4.1,-4.1) circle (0.7pt);
\draw [fill=black] (2.9,-2.9) circle (0.7pt);
\draw [fill=black] (0,0) circle (2.0pt);
\draw [fill=black] (1,0) circle (2.0pt);
\draw [fill=black] (2,0) circle (2.0pt);
\draw [fill=black] (5,0) circle (2.0pt);
\draw [fill=black] (0,-1) circle (2.0pt);
\draw [fill=black] (1,-1) circle (2.0pt);
\draw [fill=black] (2,-1) circle (2.0pt);
\draw [fill=black] (5,-1) circle (2.0pt);
\draw [fill=black] (0,-2) circle (2.0pt);
\draw [fill=black] (1,-2) circle (2.0pt);
\draw [fill=black] (2,-2) circle (2.0pt);
\draw [fill=black] (5,-2) circle (2.0pt);
\draw [fill=black] (0,-5) circle (2.0pt);
\draw [fill=black] (1,-5) circle (2.0pt);
\draw [fill=black] (2,-5) circle (2.0pt);
\draw [fill=black] (5,-5) circle (2.0pt);
\node [below] at (2.7,-5.7) { $\darr \alpha_{\omega \omega}$};
\end{tikzpicture}
\caption{The principal downsets ${\downarrow} \alpha_{21}$, ${\downarrow} \alpha_{ \omega 3}$, and ${\downarrow} \alpha_{\omega \omega}$.}
\label{fig:Z2embeddown}
\end{center}
\end{figure}
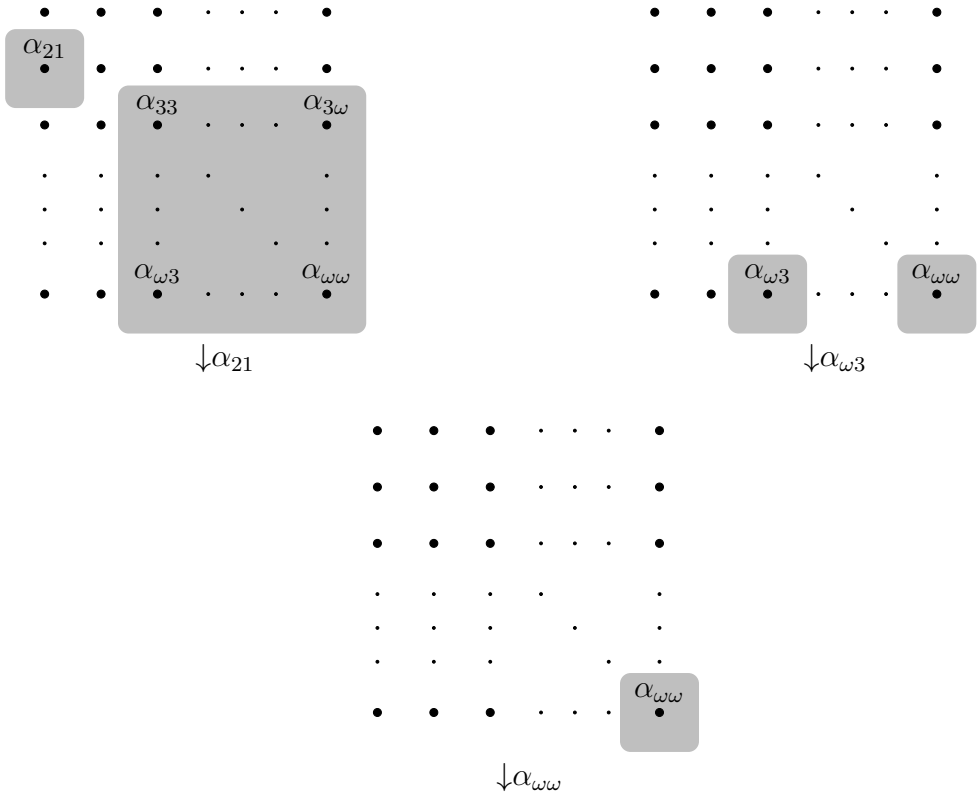

Figure~\ref{fig:Z2embeddown} shows how to calculate principal downsets
of points of $X$. From this and the description of clopens of $X$ it is straightforward to check that the downset of each
clopen is clopen.

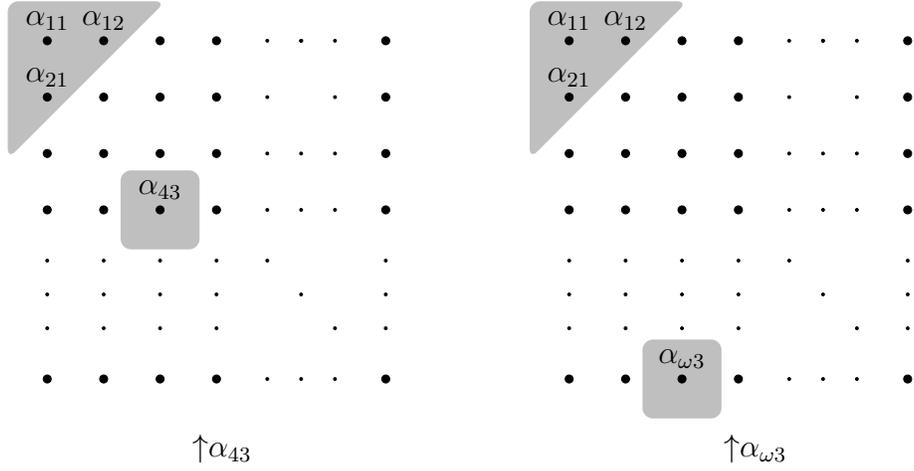
\begin{figure}[!h]
\begin{center}
\begin{tikzpicture}[scale=0.75]
\path [fill=lightgray, rounded corners] (0.3,-2.7) -- (1.7,-2.7) -- (1.7,-1.3) -- (0.3,-1.3) -- cycle;
\path [fill=lightgray, rounded corners] (-1.7,1.7) -- (-1.7,-1.1) -- (1.1,1.7) -- cycle;
\node [above] at (-1,1) { $\alpha_{1 1}$};
\node [above] at (-1,0) { $\alpha_{2 1}$};
\node [above] at (0,1) { $\alpha_{1 2}$};
\node [above] at (1,-2) { $\alpha_{4 3}$};
\draw [fill=black] (3.5,1) circle (0.7pt);
\draw [fill=black] (4.1,1) circle (0.7pt);
\draw [fill=black] (2.9,1) circle (0.7pt);
\draw [fill=black] (3.5,1) circle (0.7pt);
\draw [fill=black] (-1,-3.5) circle (0.7pt);
\draw [fill=black] (-1,-4.1) circle (0.7pt);
\draw [fill=black] (-1,-2.9) circle (0.7pt);
\draw [fill=black] (4.1,0) circle (0.7pt);
\draw [fill=black] (2.9,0) circle (0.7pt);
\draw [fill=black] (3.5,-1) circle (0.7pt);
\draw [fill=black] (4.1,-1) circle (0.7pt);
\draw [fill=black] (2.9,-1) circle (0.7pt);
\draw [fill=black] (3.5,-2) circle (0.7pt);
\draw [fill=black] (4.1,-2) circle (0.7pt);
\draw [fill=black] (2.9,-2) circle (0.7pt);
\draw [fill=black] (3.5,-5) circle (0.7pt);
\draw [fill=black] (4.1,-5) circle (0.7pt);
\draw [fill=black] (2.9,-5) circle (0.7pt);
\draw [fill=black] (0,-3.5) circle (0.7pt);
\draw [fill=black] (0,-4.1) circle (0.7pt);
\draw [fill=black] (0,-2.9) circle (0.7pt);
\draw [fill=black] (1,-3.5) circle (0.7pt);
\draw [fill=black] (1,-4.1) circle (0.7pt);
\draw [fill=black] (1,-2.9) circle (0.7pt);
\draw [fill=black] (2,-3.5) circle (0.7pt);
\draw [fill=black] (2,-4.1) circle (0.7pt);
\draw [fill=black] (2,-2.9) circle (0.7pt);
\draw [fill=black] (5,-3.5) circle (0.7pt);
\draw [fill=black] (5,-4.1) circle (0.7pt);
\draw [fill=black] (5,-2.9) circle (0.7pt);
\draw [fill=black] (3.5,-3.5) circle (0.7pt);
\draw [fill=black] (4.1,-4.1) circle (0.7pt);
\draw [fill=black] (2.9,-2.9) circle (0.7pt);
\draw [fill=black] (0,0) circle (2.0pt);
\draw [fill=black] (-1,1) circle (2.0pt);
\draw [fill=black] (0,1) circle (2.0pt);
\draw [fill=black] (1,1) circle (2.0pt);
\draw [fill=black] (2,1) circle (2.0pt);
\draw [fill=black] (5,1) circle (2.0pt);
\draw [fill=black] (-1,0) circle (2.0pt);
\draw [fill=black] (-1,-1) circle (2.0pt);
\draw [fill=black] (-1,-2) circle (2.0pt);
\draw [fill=black] (-1,-5) circle (2.0pt);
\draw [fill=black] (1,0) circle (2.0pt);
\draw [fill=black] (2,0) circle (2.0pt);
\draw [fill=black] (5,0) circle (2.0pt);
\draw [fill=black] (0,-1) circle (2.0pt);
\draw [fill=black] (1,-1) circle (2.0pt);
\draw [fill=black] (2,-1) circle (2.0pt);
\draw [fill=black] (5,-1) circle (2.0pt);
\draw [fill=black] (0,-2) circle (2.0pt);
\draw [fill=black] (1,-2) circle (2.0pt);
\draw [fill=black] (2,-2) circle (2.0pt);
\draw [fill=black] (5,-2) circle (2.0pt);
\draw [fill=black] (0,-5) circle (2.0pt);
\draw [fill=black] (1,-5) circle (2.0pt);
\draw [fill=black] (2,-5) circle (2.0pt);
\draw [fill=black] (5,-5) circle (2.0pt);
\node [above] at (5,-5.5) {};
\node [below] at (2.1,-5.8) { $\uarr \alpha_{43}$};
\end{tikzpicture}
\hspace{1.5cm}
\begin{tikzpicture}[scale=0.75]
\path [fill=lightgray, rounded corners] (0.3,-5.7) -- (1.7,-5.7) -- (1.7,-4.3) -- (0.3,-4.3) -- cycle;
\path [fill=lightgray, rounded corners] (-1.7,1.7) -- (-1.7,-1.1) -- (1.1,1.7) -- cycle;
\node [above] at (-1,1) { $\alpha_{1 1}$};
\node [above] at (-1,0) { $\alpha_{2 1}$};
\node [above] at (0,1) { $\alpha_{1 2}$};
\node [above] at (1,-5) { $\alpha_{\omega 3}$};
\draw [fill=black] (3.5,1) circle (0.7pt);
\draw [fill=black] (4.1,1) circle (0.7pt);
\draw [fill=black] (2.9,1) circle (0.7pt);
\draw [fill=black] (3.5,1) circle (0.7pt);
\draw [fill=black] (-1,-3.5) circle (0.7pt);
\draw [fill=black] (-1,-4.1) circle (0.7pt);
\draw [fill=black] (-1,-2.9) circle (0.7pt);
\draw [fill=black] (4.1,0) circle (0.7pt);
\draw [fill=black] (2.9,0) circle (0.7pt);
\draw [fill=black] (3.5,-1) circle (0.7pt);
\draw [fill=black] (4.1,-1) circle (0.7pt);
\draw [fill=black] (2.9,-1) circle (0.7pt);
\draw [fill=black] (3.5,-2) circle (0.7pt);
\draw [fill=black] (4.1,-2) circle (0.7pt);
\draw [fill=black] (2.9,-2) circle (0.7pt);
\draw [fill=black] (3.5,-5) circle (0.7pt);
\draw [fill=black] (4.1,-5) circle (0.7pt);
\draw [fill=black] (2.9,-5) circle (0.7pt);
\draw [fill=black] (0,-3.5) circle (0.7pt);
\draw [fill=black] (0,-4.1) circle (0.7pt);
\draw [fill=black] (0,-2.9) circle (0.7pt);
\draw [fill=black] (1,-3.5) circle (0.7pt);
\draw [fill=black] (1,-4.1) circle (0.7pt);
\draw [fill=black] (1,-2.9) circle (0.7pt);
\draw [fill=black] (2,-3.5) circle (0.7pt);
\draw [fill=black] (2,-4.1) circle (0.7pt);
\draw [fill=black] (2,-2.9) circle (0.7pt);
\draw [fill=black] (5,-3.5) circle (0.7pt);
\draw [fill=black] (5,-4.1) circle (0.7pt);
\draw [fill=black] (5,-2.9) circle (0.7pt);
\draw [fill=black] (3.5,-3.5) circle (0.7pt);
\draw [fill=black] (4.1,-4.1) circle (0.7pt);
\draw [fill=black] (2.9,-2.9) circle (0.7pt);
\draw [fill=black] (0,0) circle (2.0pt);
\draw [fill=black] (-1,1) circle (2.0pt);
\draw [fill=black] (0,1) circle (2.0pt);
\draw [fill=black] (1,1) circle (2.0pt);
\draw [fill=black] (2,1) circle (2.0pt);
\draw [fill=black] (5,1) circle (2.0pt);
\draw [fill=black] (-1,0) circle (2.0pt);
\draw [fill=black] (-1,-1) circle (2.0pt);
\draw [fill=black] (-1,-2) circle (2.0pt);
\draw [fill=black] (-1,-5) circle (2.0pt);
\draw [fill=black] (1,0) circle (2.0pt);
\draw [fill=black] (2,0) circle (2.0pt);
\draw [fill=black] (5,0) circle (2.0pt);
\draw [fill=black] (0,-1) circle (2.0pt);
\draw [fill=black] (1,-1) circle (2.0pt);
\draw [fill=black] (2,-1) circle (2.0pt);
\draw [fill=black] (5,-1) circle (2.0pt);
\draw [fill=black] (0,-2) circle (2.0pt);
\draw [fill=black] (1,-2) circle (2.0pt);
\draw [fill=black] (2,-2) circle (2.0pt);
\draw [fill=black] (5,-2) circle (2.0pt);
\draw [fill=black] (0,-5) circle (2.0pt);
\draw [fill=black] (1,-5) circle (2.0pt);
\draw [fill=black] (2,-5) circle (2.0pt);
\draw [fill=black] (5,-5) circle (2.0pt);
\node [below] at (2.3,-5.8) { $\uarr \alpha_{\omega 3}$};
\end{tikzpicture}
\caption{The principal upsets ${\uparrow} \alpha_{43}$ and ${\uparrow} \alpha_{\omega 3}$.}
\label{fig:Z2embedup}
\end{center}
\end{figure}

Figure~\ref{fig:Z2embedup} describes how to calculate principal upsets of points of $X$. From this and the fact that $\alpha_{\omega \omega}$
is the least element of $X$, it is easy to see that the upset of each point of $X$ is closed. Thus, $X$ is an Esakia space
(see \cite{esa74}). We can embed $Z_2$ into $X$ via the map defined by $y \mapsto \alpha_{1 \omega}$,
$z_i \mapsto \alpha_{\omega \, i+1}$, and $x \mapsto \alpha_{\omega \omega}$; see Figure~\ref{fig:Z2embedembed}.

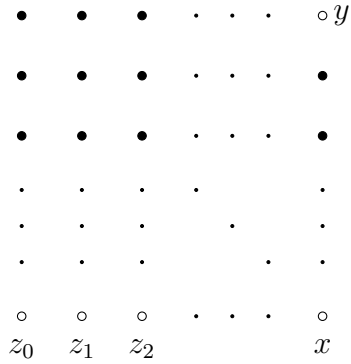
\begin{figure}[!h]
\begin{center}
\begin{tikzpicture}[scale=0.8]
\draw [fill=black] (3.5,0) circle (0.7pt);
\draw [fill=black] (4.1,0) circle (0.7pt);
\draw [fill=black] (2.9,0) circle (0.7pt);
\draw [fill=black] (3.5,-1) circle (0.7pt);
\draw [fill=black] (4.1,-1) circle (0.7pt);
\draw [fill=black] (2.9,-1) circle (0.7pt);
\draw [fill=black] (3.5,-2) circle (0.7pt);
\draw [fill=black] (4.1,-2) circle (0.7pt);
\draw [fill=black] (2.9,-2) circle (0.7pt);
\draw [fill=black] (3.5,-5) circle (0.7pt);
\draw [fill=black] (4.1,-5) circle (0.7pt);
\draw [fill=black] (2.9,-5) circle (0.7pt);
\draw [fill=black] (0,-3.5) circle (0.7pt);
\draw [fill=black] (0,-4.1) circle (0.7pt);
\draw [fill=black] (0,-2.9) circle (0.7pt);
\draw [fill=black] (1,-3.5) circle (0.7pt);
\draw [fill=black] (1,-4.1) circle (0.7pt);
\draw [fill=black] (1,-2.9) circle (0.7pt);
\draw [fill=black] (2,-3.5) circle (0.7pt);
\draw [fill=black] (2,-4.1) circle (0.7pt);
\draw [fill=black] (2,-2.9) circle (0.7pt);
\draw [fill=black] (5,-3.5) circle (0.7pt);
\draw [fill=black] (5,-4.1) circle (0.7pt);
\draw [fill=black] (5,-2.9) circle (0.7pt);
\draw [fill=black] (3.5,-3.5) circle (0.7pt);
\draw [fill=black] (4.1,-4.1) circle (0.7pt);
\draw [fill=black] (2.9,-2.9) circle (0.7pt);
\draw [fill=black] (0,0) circle (2.0pt);
\draw [fill=black] (1,0) circle (2.0pt);
\draw [fill=black] (2,0) circle (2.0pt);
\draw [fill=white] (5,0) circle (2.0pt);
\draw [fill=black] (0,-1) circle (2.0pt);
\draw [fill=black] (1,-1) circle (2.0pt);
\draw [fill=black] (2,-1) circle (2.0pt);
\draw [fill=black] (5,-1) circle (2.0pt);
\draw [fill=black] (0,-2) circle (2.0pt);
\draw [fill=black] (1,-2) circle (2.0pt);
\draw [fill=black] (2,-2) circle (2.0pt);
\draw [fill=black] (5,-2) circle (2.0pt);
\draw [fill=white] (0,-5) circle (2.0pt);
\draw [fill=white] (1,-5) circle (2.0pt);
\draw [fill=white] (2,-5) circle (2.0pt);
\draw [fill=white] (5,-5) circle (2.0pt);
\node [below] at (0,-5.2) {$z_0$};
\node [below] at (1,-5.2) {$z_1$};
\node [below] at (2,-5.2) {$z_2$};
\node [below] at (5,-5.2) {$x$};
\node [right] at (5,0) {$y$};
\end{tikzpicture}
\caption{The embedding of $Z_2$ into $X$.}
\label{fig:Z2embedembed}
\end{center}
\end{figure}
\end{example}

\section*{Acknowledgment}
We are grateful to the referee for careful reading and the comments which have improved the presentation of the paper.

\end{document}